\documentclass[smallextended]{svjour3}
\usepackage[utf8]{inputenc}
\usepackage{cite}
\usepackage{amsmath}
\usepackage{amsfonts}
\usepackage{amssymb}
\usepackage{subfigure}
\usepackage{enumerate}
\usepackage{hyperref}
\usepackage[colorinlistoftodos]{todonotes}
\usepackage{fullpage}
\usepackage{float}

\newcommand{\itrpp}{{\rm itr}_p[A,B](\bx)}

\newcommand {\B} {\mathbb B}

\newcommand{\al}{\alpha}
\newcommand{\be}{\beta}

\newcommand{\la}{\lambda}
\newcommand{\de}{\delta}
\newcommand{\eps}{\varepsilon}
\newcommand{\bx}{\bar x}

\newcommand {\bd} {{\rm bd}\,}

\newcommand{\tr}{{\rm tr}[A,B](\bx)}
\newcommand{\str}{{\rm str}[A,B](\bx)}
\newcommand{\itr}{{\rm itr}[A,B](\bx)}
\newcommand{\itrc}{{\rm itr}_c[A,B](\bx)}

\newcommand{\itrd}[1]{{\rm itr}_{#1}[A,B](\bx)}

\newcommand{\ang}[1]{\left\langle #1 \right\rangle}
\newcommand{\qdtx}[1]{\quad\mbox{#1}\quad}

\smartqed

\renewcommand {\theenumi} {\roman{enumi}}

\renewcommand{\equiv}{:=}
\newcommand{\Ball}{{\mathbb{B}}}

\newcommand{\paren}[1]{\left(#1\right)}
\newcommand{\parenn}[1]{\left[#1\right]}

\newcommand{\set}[2]{\left\{#1\,\left|\,#2\right.\right\}}
\newcommand{\ip}[2]{\left\langle #1,#2\right\rangle}
\newcommand{\norm}[1]{\left\|#1\right\|}

\newcommand{\mmap}[3]{#1:\,#2\rightrightarrows #3\,}

\newcommand {\Limsup} {\mathop{{\rm Lim\,sup}\,}}

\DeclareMathOperator{\dist}{dist}

\DeclareMathOperator{\cone}{{cone}}

\newcommand{\ncone}[1]{{N}_{#1}}

\newcommand{\lncone}[1]{{\overline{N}}_{#1}}

\title{Some new characterizations of intrinsic transversality in Hilbert spaces
\thanks{NHT and MV are supported by the European Research Council under the European Union's Seventh Framework Programme (FP7/2007-2013)/ERC grant agreement No. 339681.
THB is supported by the Australian Research Council, project DP160100854.
NDC is supported by an Australian Government Research Training Program Fee Off-Set Scholarship and a CIAO PhD Research Scholarship
through Federation University Australia.
}}
\author{Nguyen Hieu Thao, Thi Hoa Bui, Nguyen Duy Cuong, Michel Verhaegen}
\institute{Nguyen Hieu Thao \at
	Delft Center for Systems and Control, Delft University of Technology, 2628CD Delft, The Netherlands.
	Department of Mathematics, Teacher College, Can Tho University, Can Tho City, Vietnam.\\
	\email{h.t.nguyen-3@tudelft.nl,\ nhthao@ctu.edu.vn}
\and	
Thi Hoa Bui \at
Centre for Informatics and Applied Optimization, Federation University Australia, POB 663, Ballarat, VIC 3350, Australia.\\
\email{h.bui@federation.edu.au}
\and 
Nguyen Duy Cuong \at
Centre for Informatics and Applied Optimization, Federation University Australia, POB 663, Ballarat, VIC 3350, Australia.
Department of Mathematics, College of Natural Sciences, Can Tho University, Can Tho City, Vietnam.\\
\email{	duynguyen@students.federation.edu.au, ndcuong@ctu.edu.vn}
\and
Michel Verhaegen \at
Delft Center for Systems and Control,
Delft University of Technology,
2628CD Delft, The Netherlands.\\
\email{m.verhaegen@tudelft.nl}
}
\dedication{Dedicated to Professor Alexander Ya Kruger on the occasion of his 65$^{\,th}$ birthday}
\date{Received: date / Accepted: date}
\begin{document}

\maketitle

\begin{abstract}
Motivated by a number of research questions concerning transversality-type properties of pairs of sets recently raised by Ioffe \cite{Iof17.1} and Kruger \cite{Kru18}, this paper reports several new characterizations of the intrinsic transversality property in Hilbert spaces.
Our dual space results clarify the picture of intrinsic transversality, its variants and the only existing sufficient dual condition for subtransversality, and actually unify them.
New primal space characterizations of the intrinsic transversality which is originally a dual space condition lead to new understanding of the property in terms of primal space elements for the first time.
As a consequence, the obtained analysis allows us to address a number of research questions asked by the two aforementioned researchers about the intrinsic transversality property in the Hilbert space setting.


\end{abstract}

\section{Introduction}

Transversality and subtransversality are the two important properties of collections of sets which reflect the mutual arrangement of the sets around the reference point in normed spaces.
These properties are widely known as \emph{constraint qualification conditions} in optimization and variational analysis for formulating optimality conditions \cite{Mor06.1,Iof17.1,YeYe97,NgaThe01} and calculus rules for subdifferentials, normal cones and coderivatives \cite{Mor06.1,NgaThe01,KruLop12.1,KruLop12.2,Iof17,Iof00, Iof16,Iof16.2},
and as \emph{key ingredients} for establishing sufficient and/or necessary conditions for linear convergence of computational algorithms \cite{BauBor96,LewMal08,LewLukMal09,HesLuk13,LukTebNgu18,KruLukNgu18,Pha16,Tha18,DruLew18}.
We refer the reader to the papers \cite{Kru05,Kru06,Kru09,KruTha13,KruLop12.1,KruLop12.2,KruTha15,KruLukNgu17,KruLukNgu18} 
by Kruger and his collaborators for a variety of their sufficient and/or necessary conditions in both primal and dual spaces.

Transversality is strictly stronger than subtransversality, in fact, the former property is sufficient for many applications where the latter one is not, for example, in proving linear convergence of the alternating projection method for solving nonconvex feasibility problems \cite{LewMal08,LewLukMal09}, or in establishing error bounds for the Douglas-Rachford algorithm \cite{HesLuk13,Pha16} and its modified variants \cite{Tha18}.
However, transversality seems to be too restrictive for many applications, and in fact there have been a number of attempts devoted to research for weaker properties but still sufficient for the application of specific interest.
Of course, there would not exist any universal transversality-type property that works well for all applications.
When formulating necessary optimality conditions for optimization problems in terms of abstract Lagrange multipliers and establishing intersection rules for tangent cones in Banach spaces, Bivas et al. \cite{BivKraRib18} recently introduced a property called \emph{tangential transversality}, which is a primal space property lying between transversality and subtransversality.
When establishing linear convergence criteria of the alternating projection algorithm for solving nonconvex feasibility problems, a series of meaningful transversality-type properties have been introduced and analyzed in the literature: \emph{affine-hull transversality} \cite{Pha16}, \emph{inherent transversality} \cite{BauLukPhaWan13a}, \emph{separable intersection property} \cite{NolRon16} and \emph{intrinsic transversality} \cite{DruIofLew15}.
Compared to \emph{tangential transversality}, the latter ones are dual space properties since they are defined in terms of normal vectors.
Compared to the transversality property, all the above transversality-type properties are not dependent on the underlying space, that is, if a property is satisfied in the ambient space $X$, then so is it in any ambient space containing $X$.
Recall that in the finite dimensional setting, a pair of two closed sets $\{A,B\}$ is transversal at a common point $\bar{x}$ if and only if
\begin{equation}\label{tran_dual}
\lncone{A}(\bar{x}) \cap \paren{-\lncone{B}(\bar{x})} = \{0\},
\end{equation}
where $\lncone{A}(\bar{x})$ stands for the limiting normal cone to $A$ at $\bx$, see \eqref{LNC} for the definition.
This characterization reveals that transversality is a property that involves all the limiting normals to the sets at the reference point.
This fundamentally explains why the property is not invariant with respect to the ambient space as well as becomes too restrictive for many applications.  
Indeed, the hidden idea leading to the introduction of the above dual space transversality-type properties in the context of nonconvex alternating projections is simply based on the observation that not all such normal vectors are relevant for analyzing convergence of the algorithm.
\emph{Affine-hull transversality} is merely transversality but considered only in the affine hull $L$ of the two sets, that is, the pair of translated sets $\{A-\bar{x},B-\bar{x}\}$ is transversal at $0$ in the subspace $L-\bar{x}$.
As a consequence, the analysis of this property is straightforwardly obtained from that of transversality \cite{Pha16}.
The key feature of \emph{inherent transversality}\footnote{The property originated in \cite[Theorem 2.13]{BauLukPhaWan13a} without a name, then was refined and termed as \emph{inherent transversality} in Definition 4.4 of the reprint ``Drusvyatskiy, D., Ioffe, A.D., Lewis, A.S.: Alternating projections and coupling slope. Preprint, arXiv:1401.7569, 1--17 (2014)''.} 
\cite{BauLukPhaWan13a} is the use of \emph{restricted normal cones} in place of the conventional normal cones in characterization \eqref{tran_dual} of transversality in Euclidean spaces.
As a result, the analysis of this property is reduced to the calculus of the restricted normal vectors as conducted in \cite{BauLukPhaWan13a}.
The \emph{separable intersection property} \cite[Definition 1]{NolRon16} was motivated by nonconvex alternating projections and ultimately designed for this algorithm, and hence it seems to have significant impact only in this context. 
\emph{Intrinsic transversality} was also introduced in the context of nonconvex alternating projections in Euclidean spaces \cite{DruIofLew15}, it turns out to be an important property itself in variational analysis as demonstrated by Ioffe \cite[Section 9.2]{Iof17} and \cite{Iof17.1} and Kruger \cite{Kru18}.
On the one hand, a variety of characterizations of intrinsic transversality in various settings (Euclidean, Hilbert, Asplund, Banach and normed spaces) have been established by a number of researchers \cite{DruIofLew15,NolRon16,KruTha16,KruLukNgu18,Iof17,Iof17.1,Kru18}.
On the other hand, there are still a number of important research questions about this property, for example, the open questions 1-6 asked by Kruger in \cite[page 140]{Kru18} or the challenge by Ioffe about primal counterparts of intrinsic transversality \cite[page 358]{Iof17.1}.
It is known that for pairs of closed and convex sets in Euclidean spaces, the only existing sufficient dual condition for subtransversality is also a necessary condition, and it is equivalent to intrinsic transversality \cite{Kru18}.
Another interesting question is whether the latter two dual properties is equivalent in the nonconvex setting.

Motivated mainly by the above research questions, this paper is devoted to investigate further primal and dual characterizations of intrinsic transversality in connection with related properties.
Apart from the appeal to address the above research questions, this work was also motivated by the potential for meaningful applications of these transversality-type properties, for example, in establishing convergence criteria for more involved projection algorithms (rather than the alternating projection method) and in formulating calculus rules for \emph{relative normal cones} (see Definition \ref{D2}).

The organization of the paper is as follows.
New dual space results in the Hilbert space setting are presented in Section \ref{s:Tran,subtran,itr tran} with the key estimate obtained in Theorem \ref{t:question 3}.
We show the equivalence between the only existing sufficient dual condition for subtransversality \cite[Theorem 1]{KruLukNgu17} and the intrinsic transversality property, and provide a refined characterization of the properties, Corollary \ref{c:equivalence} and Corollary \ref{c:suff_subtran}.
These results significantly clarify the picture of intrinsic transversality, its variants introduced and analyzed in \cite{Kru18} and sufficient dual conditions for subtransversality, and actually unify them.
As a consequence, we address an important research question asked by Kruger in \cite[question 3, page 140]{Kru18} in the Hilbert space setting.
The analysis of intrinsic transversality in finite dimensional spaces is presented in Section \ref{s:finite dimensional}.
The notions of \emph{(restricted) relative limiting normal cones} \cite[Definition 2]{Kru18}, which themselves can also be of interest, were introduced and proved to be useful in characterizing transversality-type properties in \cite{Kru18}.
We prove that the two cones are equal when restricted on the constraint set $C$ given by \eqref{con_cone} only elements in which are of interest for characterizing transversality-type properties, Theorem \ref{t:question 4}.
This new understanding of these normal objects in turn yields insights about intrinsic transversality.
As a consequence, we address another important research question asked by Kruger in \cite[question 4, page 140]{Kru18} in the Euclidean space setting.
New primal space results in the Hilbert space setting are presented in Section \ref{s:primal}.
We formulate for the first time a primal space characterization of intrinsic transversality, Theorem \ref{t:primal result}.
These results, which substantially open perspective to view intrinsic transversality from primal space elements, were motivated by the research challenge raised by Ioffe in \cite[page 358]{Iof17.1}.
\bigskip 

Our basic notation is standard; cf. \cite{Mor06.1,RocWet98,DonRoc14}.
The setting throughout the current paper is a Hilbert space $X$.
In order to clearly distinguish elements in the primal space from those in the dual space, we also denote $X^*$ the topological dual of $X$ and
$\langle\cdot,\cdot\rangle$ the bilinear form defining the pairing between the two spaces.
The open unit balls in $X$ and $X^*$ are denoted by $\B$ and $\B^*$, respectively, and $\B_\delta(x)$ (respectively, $\overline{\B}_\delta(x)$) stands for the open (respectively, closed) ball with center $x$ and radius $\delta>0$.
The distance from a point $x \in X$ to a set $\Omega \subset X$ is defined by $\dist(x, \Omega):=\inf_{\omega\in \Omega}\|x-\omega\|$, and we use the convention $\dist(x,\Omega) = +\infty$ when $\Omega = \emptyset$.
The set-valued mapping
\begin{equation*}
\mmap{P_{\Omega}}{X}{X}\colon
x\mapsto \set{\omega\in \Omega}{\norm{x-\omega} = \dist(x,\Omega)}
\end{equation*}
is the \emph{projector} on $\Omega$.
An element $\omega\in P_{\Omega}(x)$ is called a {\em projection}.
This exists for any $x\in X$ and any closed set $\Omega\subset X$.
Note that the projector is not, in general, single-valued.
If $\Omega$ is closed and convex, then $P_\Omega$ is singleton everywhere.
The \emph{inverse of the projector}, $P^{-1}_{\Omega}$, is defined by
\begin{equation*}
P^{-1}_{\Omega}(\omega)\equiv\set{x\in X}{\omega \in P_{\Omega}(x)}\quad \forall \omega \in \Omega.
\end{equation*}
The proximal normal cone to $\Omega$ at a point $\bx \in \Omega$ is defined by
\begin{gather*}
N^p_{\Omega}(\bx) := \cone\left(P_{\Omega}^{-1}(\bx) - \bx \right),
\end{gather*}
which is a convex cone.
Here $\cone(\cdot)$ denotes the smallest cone containing the set within the brackets.\\
The \emph{Fr\'echet normal cone} to $\Omega$ at $\bx$ is defined by (cf. \cite{Kru03})
\begin{gather}\label{NC}
	N_{\Omega}(\bx):= \left\{x^\ast\in X^\ast\mid
	\limsup_{x\stackrel{\Omega}{\rightarrow}\bar x, x\neq \bar{x}} \frac {\langle x^\ast,x-\bx\rangle}
	{\|x-\bx\|} \le 0 \right\},
\end{gather}
which is a nonempty closed convex cone.
Here $x\stackrel{\Omega}{\rightarrow}\bar x$ means $x \rightarrow \bar x$ and $x \in \Omega$.\\
The \emph{limiting normal cone} to $\Omega$ at $\bx$ is defined by
		\begin{align}\label{LNC}
		\overline{N}_{\Omega}(\bar x):= \Limsup_{x\stackrel{\Omega}{\rightarrow}\bar x}N_{\Omega}(x):=\left\{x^*=\lim_{k\to\infty}x^*_k\mid x^*_k\in N_{\Omega}(x_k),\;x_k\in \Omega,\;x_k\to\bx\right\}.
		\end{align}	
In the above definition, the \emph{Fr\'echet} normal cone can equivalently be replaced by the \emph{proximal} one.
It holds that $N^p_{\Omega}(\bx)
\subset N_{\Omega}(\bx) \subset \overline{N}_{\Omega}(\bar x)$ and that if $\Omega$ is closed and $\dim X<\infty$, then $\overline{N}_{\Omega}(\bar x) \ne \{0\}$ if and only if $\bx \in \bd \Omega$.
By convention, we define $N^p_{\Omega}(\bx) = N_{\Omega}(\bx) = \overline{N}_{\Omega}(\bar x) = \emptyset$ whenever $\bx\notin \Omega$.
If $\Omega$ is a convex set, then all the above normal cones coincide and reduce to the one in the sense of convex analysis (e.g., \cite[Proposition~2.4.4]{Cla83}, \cite[Proposition~1.19]{Kru03})
	\begin{gather*}\label{CNC}
	N^p_{\Omega}(\bx) = N_{\Omega}(\bx) = \overline{N}_{\Omega}(\bar x) = \left\{x^*\in X^*\mid \langle x^*,x-\bx \rangle \leq 0 \qdtx{for all} x\in \Omega\right\}.
	\end{gather*}

\section{Transversality, Subtransversality and Intrinsic Transversality}\label{s:Tran,subtran,itr tran}

The following definition recalls possibly the most widely known regularity properties of pairs of sets.

\begin{definition}\label{D3}
Let $\{A,B\}$ be a pair of sets in $X$, and $\bx\in A\cap B$.
\begin{enumerate}
\item
$\{A,B\}$ is \emph{subtransversal} at $\bar x$ if there exist numbers $\alpha\in]0,1]$ and $\delta>0$ such that
\begin{align}\label{D1-1}
\alpha \dist\left(x,A\cap B\right)\le \max\left\{\dist(x,A),\dist(x,B)\right\}\quad \forall x\in \B_{\delta}(\bar{x}).
\end{align}
\item
$\{A,B\}$ is \emph{transversal} at $\bar x$ if there exist numbers $\alpha\in]0,1]$ and $\delta>0$ such that
\begin{align}\label{D1-2}
\alpha \dist\left(x,(A-x_1)\cap (B-x_2)\right)\le \max\left\{\dist(x,A-x_1), \dist(x,B-x_2)\right\}\quad \forall x\in \B_{\delta}(\bar{x}),\; x_1,x_2\in \delta\B.
\end{align}
\end{enumerate}
The exact upper bound of all $\alpha\in]0,1]$ such that condition \eqref{D1-1} or condition \eqref{D1-2} is satisfied for some $\de>0$ is denoted by $\str$ or $\tr$, respectively, with the convention that the supremum of the empty set equals $0$.
Then $\{A,B\}$ is subtransversal (respectively, transversal) at $\bx$ if and only if $\str > 0$ (respectively, $\tr > 0$).
It is clear that \eqref{D1-2} implies \eqref{D1-1} by simply setting $x_1 = x_2=0$.
That is, transversality is stronger than subtransversality and, as a consequence, it always holds that $\tr \le \str$.
\end{definition}

\begin{remark}
\begin{enumerate}
\item (\emph{subtransversality})
The subtransversality property was initially studied by Bauschke and Borwein \cite{BauBor93} under the name \textit{linear regularity} as a sufficient condition for linear convergence of the alternating projection algorithm for solving convex feasibility problems in Hilbert spaces.
Their results were later extended to the cyclic projection algorithm for solving feasibility problems involving a finite number of convex sets \cite{BauBor96}.
The term of \textit{linear regularity} was widely adapted in the community of variational analysis and optimization for several decades, for example, Bakan et al. \cite{BakDeuLi05}, Bauschke et al. \cite{BauBorLi99,BauBorTse00}, Li et al. \cite{LiNgPon07}, Ng and Zang \cite{NgZan07}, Zheng and Ng \cite{ZheNg08}, Kruger and his collaborators \cite{Kru05,Kru06,Kru09,KruTha13,KruTha14,KruTha15,Tha15}.
In the survey \cite{Iof00}, Ioffe used the property (without a name) as a qualification condition for establishing calculus rules for normal cones and subdifferentials.
Ngai and Th\'era \cite{NgaThe01} named the property as \textit{metric inequality}  and used it to characterize the Asplund space as well as to establish calculus rules for the limiting Fr\'echet subdifferential.
Penot \cite{Pen13} referred the property as \textit{linear coherence} and applied it in formulating calculus rules for the \emph{viscosity Fr\'echet} and \emph{viscosity Hadamard} subdifferentials.
The name (sub)transversality was coined by Ioffe in the 2016 survey \cite[Definition 6.14]{Iof16}, and then he explained the philosophy for this choice in his 2017 book
\cite[page 301]{Iof17.1} that ``Regularity is a property of a single object while transversality relates to the
interaction of two or more independent objects''.
In spite of the relatively long history with many important features of subtransversality, useful applications of the property keep being discovered.
For example, Luke et al. \cite[Theorem 8]{LukTebNgu18}\footnote{The result was established in Euclidean spaces. Fortunately, its proof remains valid in the Hilbert space setting without changes.}    
very recently proved that subtransversality is not only sufficient but also necessary for linear convergence of convex alternating projections.
This complements the aforementioned result by Bauschke and Borwein \cite{BauBor93} obtained 25 years earlier.
Luke et al. \cite[Section 4]{LukTebNgu18} also revealed that the property has been imposed either explicitly or implicitly in all existing linear convergence criteria for nonconvex alternating projections, and hence conjectured that subtransversality is a necessary condition for linear convergence of the algorithm.

\item (\emph{transversality})
The origin of the concept of \emph{transversality} can be traced back to at least the 1970's \cite{GuiPol74,Hir76} in differential geometry which deals of course with smooth manifolds, where transversality of a pair of smooth manifolds $\{A,B\}$ at a common point $\bx$ can also be characterized by \eqref{tran_dual}\footnote{In this special setting, the normal cones appearing in \eqref{tran_dual} are subspaces which correspondingly coincide with the \emph{normal spaces} (i.e., orthogonal complements to the \emph{tangent spaces}) to the manifolds at $\bx$. Particularly, the minus sign in \eqref{tran_dual} can be omitted.}.
The property is known as a sufficient condition for the intersection $A\cap B$ to be also a smooth manifold around $\bx$.
To the best of our awareness, transversality of pairs (collections) of general sets in normed linear spaces was first investigated by Kruger in  a systematic picture of mutual arrangement properties of the sets.
The property has been known under quite a number of other names including \emph{regularity}, \emph{strong regularity}, \emph{property $(UR)_S$}, \emph{uniform regularity}, \emph{strong metric inequality} \cite{Kru05,Kru06,Kru09,KruTha13} and \textit{linear regular intersection} \cite{LewLukMal09}.
Plenty of primal and dual space characterizations of transversality (especially in the Euclidean space setting) as well as its close connections to important concepts in optimization and variational analysis such as \emph{weak sharp minima}, \emph{error bounds}, \emph{conditions involving primal and dual slopes}, \emph{metric regularity}, \emph{(extended) extremal principles} and other types of mutual arrangement properties of collections of sets have been established and extended to more general \emph{nonlinear settings} in a series of papers by Kruger and his collaborators \cite{KhaKruTha15,Kru15,Kru16,Kru15.2,Kru06,KruTha14,KruTha15,Kru05,Kru09}.
Apart from classical applications of the property, for example, as a sufficient condition for \emph{strong duality} to hold for convex optimization (Slater's condition) \cite{BorLew00,BoyVan04} or as a constraint qualification condition for establishing calculus rules for the limiting/Mordukhovich normal cones \cite[page 265]{Mor06.1} and coderivatives (in connection with \emph{metric regularity}, the counterpart of transversality in terms of set-valued mappings) \cite{RocWet98,DonRoc14}, important applications have also been emerging in the field of numerical analysis.
Lewis et al. \cite{LewMal08,LewLukMal09} applied the property to establish the first linear convergence criteria for nonconvex alternating and averaged projections.
Transversality was also used to prove linear convergence of the Douglas-Rachford algorithm \cite{HesLuk13,Pha16} and its relaxations \cite{Tha18}.
A practical application of these results is to the \emph{phase retrieval problem} where transversality is sufficient for linear convergence of alternating projections, the Douglas-Rachford algorithm and actually any convex combinations of the two algorithms \cite{ThaSolVer18}.
\end{enumerate}
We refer the reader to the recent surveys by Kruger et al. \cite{KruLukNgu17,KruLukNgu18} for a more comprehensive discussion about the two properties.
\end{remark}
	
A number of dual characterizations of transversality, especially in the Euclidean space setting, have been established \cite{Kru05,Kru06,Kru09,LewLukMal09,KruTha13,KruTha15,KruLukNgu18} and applied, for example, \cite{Mor06.1,LewLukMal09,Pha16,Tha18}.
The situation is very much different for subtransversality.
For collections of closed and convex sets, the following dual characterization of subtransversality is due to Kruger.

\begin{proposition}\label{p:dual_subtran_convex} \cite[Theorem 3]{Kru18}\footnote{The result is also valid in Banach spaces.}
A pair of closed and convex sets $\{A,B\}$ is subtransversal at a point $\bx\in A\cap B$ if and only if there exist numbers $\alpha\in]0,1[$ and $\delta>0$ such that $\|x^*_1+x^*_2\|>\alpha$
for all $a\in(A\setminus B)\cap \B_\de(\bx)$, $b\in(B\setminus A)\cap \B_\de(\bx)$, $x\in \B_\de(\bx)$ with $\|x-a\|=\|x-b\|$ and $x_1^*,x_2^*\in X^*$ satisfying 
\begin{gather*}
\dist(x_1^*,N_{A}(a))<\delta,\quad \dist(x_2^*,N_{B}(b))<\delta,\\
\|x^*_1\|+\|x^*_2\|=1,\quad
\langle x_1^*,x-a\rangle=\|x_1^*\|\,\|x-a\|,\;
\langle x_2^*,x-b\rangle=\|x_2^*\|\,\|x-b\|.
\end{gather*}
\end{proposition}
	
In the nonconvex setting, the first sufficient dual condition for subtransversality was formulated in \cite[Theorem 4.1]{KruTha14} following the routine of deducing metric subregularity characterizations for set-valued mappings in \cite{Kru15}.
The result was then refined successively in \cite[Theorem 4$(ii)$]{KruLukNgu18}, \cite[Theorem 2]{KruLukNgu17} and finally in \cite{Kru18} in the following form.

\begin{proposition}\label{p:suff_subtran} \cite[combination of Definition 2 and Corollary 2]{Kru18}\footnote{The result is also valid in Asplund spaces.} 
A pair of closed sets $\{A,B\}$ is subtransversal at a point $\bx\in A\cap B$ if there exist numbers $\alpha\in]0,1[$ and $\delta>0$ such that, for all $a\in(A\setminus B)\cap \B_\de(\bx)$, $b\in(B\setminus A)\cap \B_\de(\bx)$ and $x\in \B_\de(\bx)$
with $\norm{x-a}=\norm{x-b}$,
one has $\|x^*_1+x^*_2\|>\alpha$ for some $\eps>0$ and all $a'\in A\cap \B_\eps(a)$, $b'\in B\cap \B_\eps(b)$, $x_1'\in\B_\eps(a)$, $x_2'\in \B_\eps(b)$ with $\norm{x-x_1'}=\norm{x-x_2'}$, and $x_1^*,x_2^*\in X^*$ satisfying 
\begin{gather*}
\dist(x_1^*,N_{A}(a'))<\delta,\quad \dist(x_2^*,N_{B}(b'))<\delta,\\
\|x^*_1\|+\|x^*_2\|=1,\quad
\ang{x^*_1,x-x_1'}=\|x^*_1\|\|x-x_1'\|,\quad
\ang{x^*_2,x-x_2'}=\|x^*_2\|\|x-x_2'\|.
\end{gather*}
\end{proposition}
The inverse implication of Proposition \ref{p:suff_subtran} is unknown.
Our subsequent analysis particularly shows the negative answer to this question, see Remark \ref{r:suff but not ness}.

Compared to transversality and subtransversality, the intrinsic transversality property below appeared very recently.

\begin{definition}\label{d:intr_tran:DIL} \cite[Definition 3.1]{DruIofLew15}
A pair of closed sets $\{A,B\}$ in a Euclidean space is \emph{intrinsically transversal} at a point $\bx \in A \cap B$ if there exists an angle $\al > 0$ together with a number $\de>0$ such that any two points $a \in (A\setminus B) \cap \B_{\de}(\bx)$
and $b \in (B\setminus A) \cap \B_{\de}(\bx)$ cannot have difference $a-b$ simultaneously making an angle strictly less than $\al$ with both the proximal normal cones $N^p_B(b)$ and $-N^p_A(a)$.
\end{definition}

The above property was originally introduced in 2015 by Drusvyatskiy et al. \cite{DruIofLew15} as a sufficient condition for establishing local linear convergence of the alternating projection algorithm for solving nonconvex feasibility problems in Euclidean spaces.
As demonstrated by Ioffe \cite{Iof17}, Kruger et al. \cite{KruLukNgu17,Kru18} and will also be in this paper, intrinsic transversality turns out to be an important property itself in the field of variational anyalysis.
Kruger \cite{Kru18} recently extended and investiaged intrinsic transversality in more general underlying spaces.

\begin{definition}\label{d:intr_tran} 
\cite[Definition 2$(ii)$]{Kru18}\footnote{The property was defined and investigated in general normed linear spaces.}
A pair of closed sets $\{A,B\}$ is \emph{intrinsically transversal} at a point $\bx \in A \cap B$ if there exist numbers $\alpha\in]0,1[$ and $\delta>0$ such that $\|x^*_1+x^*_2\|>\alpha$
for all $a\in(A\setminus B)\cap \B_\de(\bx)$, $b\in(B\setminus A)\cap \B_\de(\bx)$, $x\in \B_\de(\bx)$ with
$x\ne a$, $x\ne b$, $1-\delta<\frac{\norm{x-a}}{\norm{x-b}}<1+\delta$, and $x_1^*\in N_{A}(a)\setminus\{0\}$, $x_2^*\in N_{B}(b)\setminus\{0\}$ satisfying
\begin{gather*}
\norm{x^*_1}+\norm{x^*_2}=1,\quad\frac{\ang{x_1^*,x-a}} {\|x_1^*\|\|x-a\|} >1-\delta,\quad
\frac{\ang{x_2^*,x-b}} {\|x_2^*\|\|x-b\|} >1-\delta.
\end{gather*}
\end{definition}

It is worth noting that the extension from Definition \ref{d:intr_tran:DIL} to Definition \ref{d:intr_tran} of intrinsic transversality is nontrivial and the coincidence of the two definitions in the Euclidean space setting was shown in \cite[Proposition 8$(iii)$]{Kru18}.

It was proved in \cite[Theorem 4]{Kru18} that intrinsic transversality implies the sufficient dual condition of subtransversality provided in Proposition \ref{p:suff_subtran}, which in turn implies the one stated in Proposition \ref{p:dual_subtran_convex}.
The following quantitative constants \cite{KruLukNgu17,Kru18} respectively characterizing the three
dual space properties will be convenient for our subsequent discussion and analysis\footnote{In \cite{Kru18}, the restrictions $x\neq a$ and $x\neq b$ were also used under the $\liminf$ of \eqref{itr constant}. We note that they are redundant due to the constraints $a\in A\setminus B$, $b\in B\setminus A$ and $\frac{\norm{x-a}}{\norm{x-b}}\to1$.}

\begin{align}\label{itr constant}
\itr:=&\liminf\limits_{\substack{
a\to\bar x,\; b\to\bar x,\; x\to\bar x\\	a\in A\setminus B,\;b\in B\setminus A,\;x\ne a,\;x\ne b\\ x^*_1\in N_{A}(a)\setminus\{0\},\;x^*_2\in N_{B}(b)\setminus\{0\},\;\norm{x^*_1}+\norm{x^*_2}=1
\\
\frac{\norm{x-a}}{\norm{x-b}}\to1,\; \frac{\ang{x^*_1,x-a}}{\norm{x^*_1}\norm{x-a}}\to1,\; \frac{\ang{x^*_2,x-b}}{\norm{x^*_2}\norm{x-b}}\to1}}\|x^*_1+x^*_2\|,
\\ \nonumber 
\itrd{w}:=& \lim_{\rho\downarrow 0} \inf_{\substack{a\in(A\setminus B)\cap \B_\rho(\bx),\;b\in(B\setminus A)\cap \B_\rho(\bx)\\x\in \B_\rho(\bx),\; \norm{x-a}=\norm{x-b}}}
\\ \notag
&\liminf_{\substack{x_1'\to a,\;x_2'\to b,\; a'\to a,\;b'\to b\\a'\in A,\;b'\in B,\; \|x-x_1'\|=\|x-x_2'\|\\ \dist(x_1^*,N_{A}(a'))<\rho,\; \dist(x_2^*,N_{B}(b'))<\rho,\; \|x_1^*\|+\|x_2^*\|=1
\\ \notag
\langle x_1^*,x-x_1'\rangle=\|x_1^*\|\,\|x-x_1'\|,\;
\langle x_2^*,x-x_2'\rangle=\|x_2^*\|\,\|x-x_2'\|}}
\|x^*_1+x^*_2\|,
\\\label{itrc constant}
\itrc:=& \liminf_{\substack{x\to\bx,\;a\to\bx,\;b\to\bx\\a\in A\setminus B,\;b\in B\setminus A,\; \norm{x-a}=\norm{x-b}\\ \dist(x_1^*,N_{A}(a))\to0,\;\dist(x_2^*,N_{B}(b))\to0,\; \|x_1^*\|+\|x_2^*\|=1
\\
\langle x_1^*,x-a\rangle=\|x_1^*\|\,\|x-a\|,\;
\langle x_2^*,x-b\rangle=\|x_2^*\|\,\|x-b\|}}
\|x^*_1+x^*_2\|,
\end{align}
with the convention that the infimum over the empty set equals 1.

In terms of these constants, intrinsic transversality and Propositions \ref{p:suff_subtran} and \ref{p:dual_subtran_convex} respectively admit more concise descriptions.

\begin{proposition}
Let $\{A,B\}$ be a pair of closed sets, and $\bx\in A\cap B$.
\begin{enumerate}
\item $\{A,B\}$ is intrinsically transversal at $\bar x$ if and only if $\itr>0$.
\item
$\{A,B\}$ is subtransversal at $\bar x$ if $\itrd{w}>0$.
\item If the sets $A,B$ are convex, then $\{A,B\}$ is subtransversal at $\bar x$ if and only if $\itrc>0$.
\end{enumerate}
\end{proposition}

The quantitative relationships amongst the five characterization constants defined at Definition \ref{D3} and expressions \eqref{itr constant}--\eqref{itrc constant} are as follows.

\begin{proposition}\label{p:relation_cons} \cite[Proposition 1]{Kru18}
Let $\{A,B\}$ be a pair of closed sets and $\bx\in A\cap B$.
\begin{enumerate}
\item $0 \le \tr \le \itr \le \itrd{w} \le \itrc \le 1$.\footnote{The statement is valid in Banach spaces.}
\item $\itrd{w} \le \str$.\footnote{The statement is valid in Asplund spaces.}
\item If $A$ and $B$ are convex, then $\itrc = \str$.
\item If $\dim X<\infty$ and $A,B$ are convex, then $\itrd{w} = \itrc = \str$.
\end{enumerate}
\end{proposition}

\begin{remark}[about notation and terminology]
	It is clear from Proposition \ref{p:relation_cons}$(i)$ that the strict inequality $\itrd{w}>0$ characterizes a weaker dual property than intrinsic transversality.
	That property is indeed called \emph{weak} intrinsic transversality in \cite{KruLukNgu17,Kru18}.
	This somehow explains why the letter ``$w$'' has been used in the notation $\itrd{w}$.
	Similarly, the strict inequality $\itrc>0$ also characterizes some weaker dual property than (weak) intrinsic transversality.
	Such a property has not been named yet, and it has played an important role in the analysis of transversality-type properties mainly in the \emph{convex} setting \cite{Kru18}.
	This somehow explains why the letter ``$c$'' has been used in the notation $\itrc$.
	Since one of the main results of this paper (Corollary \ref{c:equivalence}) reveals that these two constants do coincide with $\itr$ in the Hilbert space setting and, as a result, they characterize the same property - intrinsic transversality, we choose to keep as a minimum number of terminologies as possible in this paper for clarification.
	It is worth emphasizing that in the general normed linear space setting, such a coincidence remains as a challenging open question and it is natural to treat those properties characterized by the constants $\itrd{w}$ and $\itrc$ independently and as importantly as the intrinsic transversality property, see \cite{Kru18}. 
\end{remark}	

We are now ready to formulate one of the main results of this paper.
The statement and its proof is rather technical, and its meaningful consequences will be clarified subsequently.

\begin{theorem}[key estimate]\label{t:question 3}
	Let $\{A,B\}$ be a pair of closed sets and $\bx\in A\cap B$.
It holds that
\begin{equation}\label{quan_est}
\min\left\{\itrd{c},1/\sqrt{2}\right\} \le \itr.
\end{equation}
\end{theorem}

\begin{proof}
	To proceed with the proof, let us suppose that $\itrd{c}>0$ since there is nothing to prove in the case $\itrd{c}=0$.
	Let us fix an arbitrary number
	\begin{equation}\label{be}
	\be \in \left]0,\min\left\{\itrd{c},1/\sqrt{2}\right\}\right[
	\end{equation}
	and prove that $\itr\ge \be$.
	By the definition \eqref{itrc constant} of $\itrd{c}$, there exist numbers
	\begin{equation*}
	\alpha\in \left]\be,\min\left\{\itrd{c},1/\sqrt{2}\right\}\right[
	\end{equation*}
	and $\delta>0$ such that, for all $a\in (A\setminus B)\cap \Ball_{\delta}(\bx)$, $b\in (B\setminus A)\cap \Ball_{\delta}(\bx)$ and $x\in \Ball_{\delta}(\bx)$ with $\|x-a\|=\|x-b\|$, one has
	\begin{equation}\label{>al}
	\|x^*_1+x^*_2\|>\alpha
	\end{equation}
	for all $x_1^*,x_2^*\in X^*$ satisfying
	\begin{align}\label{11}
	&\dist(x_1^*,N_A(a))<\delta,\quad \dist(x_2^*,N_B(b))<\delta,\\
	\label{13}
	\|x_1^*\|+\|x_2^*\|=1,\quad &\ip{x_1^*}{x-a} = \|x_1^*\|\|x-a\|,\quad \ip{x_2^*}{x-b} = \|x_2^*\|\|x-b\|.
	\end{align}
	Choose a number $\de'\in \,[0,\de/3[$ and satisfying
	\begin{align}\label{del'}
	&2\paren{\sqrt{\de'}+\de'} < \frac{1}{2} -\be^2,
	\\\label{del'2}
	\sqrt{2\de'} + &2\sqrt{\frac{2\de'-\de'^2}{4-6\de'+3\de'^2}} < \min\left\{\de,\al-\be\right\}.
	\end{align}
	Such a number $\de'$ exists since $\frac{1}{2} -\be^2>0$, $\min\left\{\de,\al-\be\right\}>0$ and
	\begin{equation*}
	\lim_{t\downarrow 0}2\paren{\sqrt{t}+t}=0,\;\quad  \lim_{t\downarrow 0}\paren{\sqrt{2t} + 2\sqrt{\frac{2t-t^2}{4-6t+3t^2}}}=0.
	\end{equation*}
	We are going to prove $\itr\ge \be$ with the technical constant $\de'>0$.
	To begin, let us take any $a\in (A\setminus B)\cap \Ball_{\delta'}(\bx)$, $b\in (B\setminus A)\cap \Ball_{\delta'}(\bx)$ and $x\in \Ball_{\delta'}(\bx)$ with $x\neq a$, $x\neq b$,
	\begin{equation}\label{1 mp delta}
	1-\de' <\frac{\|x-a\|}{\|x-b\|} < 1+ \de',
	\end{equation}
	and $x_1^*\in N_A(a)\setminus \{0\}$, $x_2^*\in N_B(b)\setminus \{0\}$ satisfying
	\begin{align}\label{13'}
	\|x_1^*\|+\|x_2^*\|=1,\quad \frac{\ip{x_1^*}{x-a}}{\|x_1^*\|\|x-a\|}>1-\de',\quad \frac{\ip{x_2^*}{x-b}}{\|x_2^*\|\|x-b\|}>1-\de'.
	\end{align}
	All we need is to show that
	\begin{equation}\label{conclusion}
	\|x_1^*+x_2^*\| > \be.
	\end{equation}
	We first observe from \eqref{13'} that
	\begin{equation}\label{gap_2}
	\begin{aligned}
	\norm{\frac{x_1^*}{\|x_1^*\|}-\frac{x-a}{\|x-a\|}}^2 = 2 -2\frac{\ip{x_1^*}{x-a}}{\|x_1^*\|\|x-a\|} < 2-2(1-\de') = 2\de',\\
	\norm{\frac{x_2^*}{\|x_2^*\|}-\frac{x-b}{\|x-b\|}}^2 = 2 -2\frac{\ip{x_2^*}{x-b}}{\|x_2^*\|\|x-b\|} < 2-2(1-\de') = 2\de'.
	\end{aligned}
	\end{equation}
	We take care of two possibilities concerning the value of $\ip{x-a}{x-b}$ as follows.
	
	\emph{Case 1.} $\ip{x-a}{x-b}>0$. Then
	\begin{equation*}
	\norm{\frac{x-a}{\|x-a\|}-\frac{x-b}{\|x-b\|}}^2 = 2 - 2\frac{\ip{x-a}{x-b}}{\|x-a\|\|x-b\|} < 2.
	\end{equation*}
	Equivalently,
	\begin{equation}\label{gap_1}
	\norm{\frac{x-a}{\|x-a\|}-\frac{x-b}{\|x-b\|}} < \sqrt{2}.
	\end{equation}
	By the triangle inequality and estimates \eqref{gap_1}, \eqref{gap_2}, we get that
	\begin{equation*}
	\begin{aligned}
	\norm{\frac{x_1^*}{\|x_1^*\|}-\frac{x_2^*}{\|x_2^*\|}} &\le \norm{\frac{x-a}{\|x-a\|}-\frac{x-b}{\|x-b\|}} + \norm{\frac{x_1^*}{\|x_1^*\|}-\frac{x-a}{\|x-a\|}} + \norm{\frac{x_2^*}{\|x_2^*\|}-\frac{x-b}{\|x-b\|}}\\
	&< \sqrt{2} + 2\sqrt{2\de'} = \sqrt{2}\paren{1+2\sqrt{\de'}}.
	\end{aligned}
	\end{equation*}
	This implies that
	\begin{align}\label{ip>}
	\norm{\frac{x_1^*}{\|x_1^*\|}-\frac{x_2^*}{\|x_2^*\|}}^2=2-2\frac{\ip{x_1^*}{x_2^*}}{\norm{x_1^*}\norm{x_2^*}} < 2\paren{1+2\sqrt{\de'}}^2
	\Leftrightarrow \ip{x_1^*}{x_2^*} > -4\paren{\sqrt{\de'}+\de'}\norm{x_1^*}\norm{x_2^*}.
	\end{align}
	Using $\norm{x_1^*}+\norm{x_2^*}=1$ which implies $\norm{x_1^*}\norm{x_2^*}\le 1/4$ and \eqref{ip>}, respectively, we obtain that
	\begin{equation*}
	\begin{aligned}
	\norm{x_1^*+x_2^*}^2 &=  \norm{x_1^*}^2 +\norm{x_2^*}^2+2\ip{x_1^*}{x_2^*} = 1-2 \norm{x_1^*}\norm{x_2^*}+2\ip{x_1^*}{x_2^*}\\
	&> 1-2 \norm{x_1^*}\norm{x_2^*}-8\paren{\sqrt{\de'}+\de'}\norm{x_1^*}\norm{x_2^*} \ge \frac{1}{2}-2\paren{\sqrt{\de'}+\de'}.
	\end{aligned}
	\end{equation*}
	This combining with \eqref{del'} yields that
	\begin{equation*}
	\begin{aligned}
	\norm{x_1^*+x_2^*} > \sqrt{\frac{1}{2}-2\paren{\sqrt{\de'}+\de'}}\;> \be.
	\end{aligned}
	\end{equation*}
	
	\emph{Case 2.}
	\begin{equation}\label{Case2_ass}
	\ip{x-a}{x-b}\le 0.
	\end{equation}
	Let us define $m=\frac{a+b}{2}$ and
	\begin{equation}\label{x'}
	x' = x - \frac{\ip{b-a}{x-m}}{\|b-a\|^2}(b-a).
	\end{equation}
	We first check that
	\begin{equation}\label{|x'-a|=|x'-b|}
	\|x'-a\|=\|x'-b\|.
	\end{equation}
	Indeed,
	\begin{align*}
	\|x'-a\|^2-\|x'-b\|^2 &= \|x-a\|^2-\|x-b\|^2 - 2\frac{\ip{b-a}{x-m}}{\|b-a\|^2}\ip{x-a}{b-a}\\
	&+ 2\frac{\ip{b-a}{x-m}}{\|b-a\|^2}\ip{x-b}{b-a}\\
	&= \|x-a\|^2-\|x-b\|^2 - 2\frac{\ip{b-a}{x-m}}{\|b-a\|^2}\ip{b-a}{b-a}\\
	&= \|x-a\|^2-\|x-b\|^2 - 2\ip{b-a}{x-m}\\
	&= \|x-a\|^2-\|x-b\|^2 - \ip{(x-a)-(x-b)}{(x-a) + (x-b)}=0.
	\end{align*}
	We next check that
	\begin{equation}\label{ip=0}
	\ip{x-x'}{x'-m}=0.
	\end{equation}
	Indeed, by \eqref{x'}, it holds that
	\begin{equation*}
	\ip{x-x'}{x'-m} = \frac{\ip{b-a}{x-m}}{\|b-a\|^2}\ip{b-a}{x'-m},
	\end{equation*}
	from which \eqref{ip=0} follows since
	\begin{align*}
	\ip{b-a}{x'-m} &= \ip{b-a}{x-m-\frac{\ip{b-a}{x-m}}{\|b-a\|^2}(b-a)}\\
	&= \ip{b-a}{x-m}-\frac{\ip{b-a}{x-m}}{\|b-a\|^2}\ip{b-a}{b-a}=0.
	\end{align*}
	Let us define also
	\begin{equation}\label{x^*'}
	{x_1^*}' = \frac{\|x_1^*\|}{\|x'-a\|}(x'-a),\quad   {x_2^*}' = \frac{\|x_2^*\|}{\|x'-b\|}(x'-b).
	\end{equation}
	It is clear that
	\begin{equation}\label{satisfy 13}
	\begin{aligned}
	&\|{x_1^*}'\| = \|x_1^*\|,\; \|{x_2^*}'\| = \|x_2^*\|,\; \|{x_1^*}'\| + \|{x_2^*}'\| =1,\\
	&\ip{{x_1^*}'}{x'-a} = \|{x_1^*}\|\|x'-a\| = \|{x_1^*}'\|\|x'-a\|,\\
	&\ip{{x_2^*}'}{x'-b} = \|{x_2^*}\|\|x'-b\|= \|{x_2^*}'\|\|x'-b\|.
	\end{aligned}
	\end{equation}
	We next check that
	\begin{align}\label{satisfy 11}
	\dist({x_1^*}',N_A(a)) < \de,\quad \dist({x_2^*}',N_B(b)) < \de.
	\end{align}
	Let us prove $\dist({x_1^*}',N_A(a)) < \de$.
	Indeed, since $x_1^*\in N_A(a)$ it holds by \eqref{x^*'} that
	\begin{equation}\label{key est_1}
	\begin{aligned}
	\dist({x_1^*}',N_A(a)) &\le \|{x_1^*}'-x_1^*\| = \norm{\frac{\|x_1^*\|}{\|x'-a\|}(x'-a)-x_1^*}\\
	&= \|x_1^*\|\norm{\frac{x'-a}{\|x'-a\|}-\frac{x_1^*}{\|x_1^*\|}}\\
	&\le \|x_1^*\|\paren{\norm{\frac{x-a}{\|x-a\|}-\frac{x_1^*}{\|x_1^*\|}}+\norm{\frac{x'-a}{\|x'-a\|}-\frac{x-a}{\|x-a\|}}}.
	\end{aligned}
	\end{equation}
	An upper bound of $\norm{\frac{x-a}{\|x-a\|}-\frac{x_1^*}{\|x_1^*\|}}$ has been given by \eqref{gap_2}:
	\begin{align}\label{est_term2}
	\norm{\frac{x-a}{\|x-a\|}-\frac{x_1^*}{\|x_1^*\|}} < \sqrt{2\de'}.
	\end{align}
	We now establish an upper bound of $\norm{\frac{x'-a}{\|x'-a\|}-\frac{x-a}{\|x-a\|}}$ via three steps as follows.
	
	\emph{Step 1.} We show that
	\begin{equation}\label{Step1}
	\|x-x'\|^2 \le \frac{2\de'-\de'^2}{4(1-\de')^2}\,\min\left\{\|x-a\|^2,\|x-b\|^2\right\}.
	\end{equation}
	
If $\|x-a\|\ge\|x-b\|$, then
	\begin{equation}\label{<x-m,b-a> >0}
	\begin{aligned}
	&\|x-a\|^2 - \|x-b\|^2 \ge 0\\
	\Leftrightarrow \|x-m\|^2 + \|m-a\|^2 + &2\ip{x-m}{m-a} - \|x-m\|^2 - \|m-b\|^2 - 2\ip{x-m}{m-b}\ge 0\\
	\Leftrightarrow &\ip{b-a}{x-m} \ge 0.
	\end{aligned}
	\end{equation}
	Note from \eqref{Case2_ass} that
	\begin{equation}\label{<x-b,b-a> >0}
	\ip{b-a}{x-b} = \ip{b-x}{x-b} + \ip{x-a}{x-b} = -\norm{x-b}^2 + \ip{x-a}{x-b} \le 0.
	\end{equation}
	Taking \eqref{x'}, \eqref{<x-m,b-a> >0} and \eqref{<x-b,b-a> >0} into account, we have that
	\begin{equation}\label{|x'-b|>|x-b|}
	\begin{aligned}
	\|x'-b\|^2 - \|x-b\|^2 &= \|x'-x\|^2+2\ip{x'-x}{x-b}\\
	&= \|x'-x\|^2 - 2\frac{\ip{b-a}{x-m}}{\|b-a\|^2}\ip{b-a}{x-b}\\
	&\ge \|x'-x\|^2.
	\end{aligned}
	\end{equation}
	By \eqref{|x'-a|=|x'-b|} and \eqref{ip=0} we get that
	\begin{align*}
	\|x-a\|^2+\|x-b\|^2 &= 2\|x-x'\|^2 + \|x'-a\|^2 + \|x'-b\|^2 + 2\ip{x-x'}{2x'-(a+b)} \\
	&= 2\|x-x'\|^2 + 2\|x'-b\|^2 + 4\ip{x-x'}{x'-m}\\
	&= 2\|x-x'\|^2 + 2\|x'-b\|^2.
	\end{align*}
	This together with \eqref{1 mp delta} and \eqref{|x'-b|>|x-b|} yields that
	\begin{align*}
	2\|x-x'\|^2 &= \|x-a\|^2+\|x-b\|^2 - 2\|x'-b\|^2\\
	&\le (1+\de')^2\|x-b\|^2 +\|x-b\|^2 - 2\|x'-b\|^2\\
	&\le (1+\de')^2\|x-b\|^2 +\|x-b\|^2 - 2\paren{\|x-b\|^2+\|x'-x\|^2}.
	\end{align*}
	Equivalently,
	\begin{equation}\label{Step1_case1}
	4\|x-x'\|^2 \le \paren{2\de'+\de'^2}\|x-b\|^2
	= \paren{2\de'+\de'^2}\min\{\|x-a\|^2,\|x-b\|^2\}
	\end{equation}
	since $\|x-a\|\ge \|x-b\|$ in this case.

By a similar argument, if $\|x-a\|\le \|x-b\|$, then
\[
\ip{b-a}{x-m} \le 0,\;
\ip{b-a}{x-a} \ge 0.
\]
Thus
	\begin{equation}\label{|x'-a|>|x-a|}
	\begin{aligned}
\|x'-a\|^2 - \|x-a\|^2 &= \|x'-x\|^2+2\ip{x'-x}{x-a}\\
	&= \|x'-x\|^2 - 2\frac{\ip{b-a}{x-m}}{\|b-a\|^2}\ip{b-a}{x-a}\\
	&\ge \|x'-x\|^2.
	\end{aligned}
	\end{equation}
By \eqref{|x'-a|=|x'-b|} and \eqref{ip=0} we get that
	\begin{align*}
	\|x-a\|^2+\|x-b\|^2 = 2\|x-x'\|^2 + 2\|x'-a\|^2,
	\end{align*}
which together with \eqref{1 mp delta} and \eqref{|x'-a|>|x-a|} yields that
	\begin{align*}
	2\|x-x'\|^2 \le \|x-a\|^2 + \frac{1}{(1-\de')^2}\|x-a\|^2 - 2\paren{\|x-a\|^2+\|x'-x\|^2}.
	\end{align*}
Equivalently,
	\begin{equation}\label{Step1_case2}
	4\|x-x'\|^2 \le \frac{2\de'-{\de'}^2}{(1-\de')^2}\|x-a\|^2 = \frac{2\de'-\de'^2}{(1-\de')^2}\min\{\|x-a\|^2,\|x-b\|^2\}
	\end{equation}
	since $\|x-a\|\le \|x-b\|$ in this case.
	
Combining \eqref{Step1_case1} and \eqref{Step1_case2} and noting that $2\de'+\de'^2 < \frac{2\de'-\de'^2}{(1-\de')^2}$, we obtain \eqref{Step1} as claimed.

	\emph{Step 2.} We show that
	\begin{align}\label{Step2}
	\|x'-a\|^2 \ge \frac{4-6\de'+3\de'^2}{2\de'-\de'^2}\|x-x'\|^2.
	\end{align}
		Indeed, if $\|x-a\|\le \|x-b\|$, then the use of \eqref{|x'-a|>|x-a|} and \eqref{Step1_case2} yields \eqref{Step2}:
	\begin{align*}
	 \|x'-a\|^2 &\ge \|x-a\|^2+\|x-x'\|^2\\
	&\ge \frac{4(1-\de')^2)}{2\de'-\de'^2}\|x-x'\|^2+\|x-x'\|^2 = \frac{4-6\de'+3\de'^2}{2\de'-\de'^2}\|x-x'\|^2.
	\end{align*}
		Otherwise, i.e., $\|x-a\|\ge \|x-b\|$, then the use of \eqref{|x'-a|=|x'-b|}, \eqref{|x'-b|>|x-b|} and \eqref{Step1_case1} successively implies that
	\begin{align*}
	\|x'-a\|^2 &= \|x'-b\|^2 \ge \|x-b\|^2+\|x-x'\|^2\\
	&\ge \frac{4}{2\de'+\de'^2}\|x-x'\|^2+\|x-x'\|^2 = \frac{4+2\de'+\de'^2}{2\de'+\de'^2}\|x-x'\|^2,
	\end{align*}
which also yields \eqref{Step2} since $\frac{4+2\de'+\de'^2}{2\de'+\de'^2} > \frac{4-6\de'+3\de'^2}{2\de'-\de'^2}$.
Hence \eqref{Step2} has been proved.
	
	\emph{Step 3.} We show that
	\begin{align}\label{Step3}
	\norm{\frac{x'-a}{\|x'-a\|}-\frac{x-a}{\|x-a\|}} \le 2\frac{\|x-x'\|}{\|x'-a\|}.
	\end{align}
	Indeed,
	\begin{align*}
	\norm{\frac{x'-a}{\|x'-a\|}-\frac{x-a}{\|x-a\|}} &\le \norm{\frac{x'-a}{\|x'-a\|}-\frac{x-a}{\|x'-a\|}} + \norm{\frac{x-a}{\|x'-a\|}-\frac{x-a}{\|x-a\|}}\\
	&= \frac{\|x-x'\|}{\|x'-a\|} + \left|\frac{\|x-a\|}{\|x'-a\|}-1\right|.
	\end{align*}
	If $\|x-a\|\ge \|x'-a\|$, then \eqref{Step3} holds true since
	\begin{align*}
	\left|\frac{\|x-a\|}{\|x'-a\|}-1\right| = \frac{\|x-a\|}{\|x'-a\|}-1\le  \frac{\|x-x'\|+\|x'-a\|}{\|x'-a\|}-1
	=\frac{\|x-x'\|}{\|x'-a\|}.
	\end{align*}
	Otherwise, i.e., $\|x-a\|<\|x'-a\|$, then \eqref{Step3} also holds true since
	\begin{align*}
	\left|\frac{\|x-a\|}{\|x'-a\|}-1\right| = 1- \frac{\|x-a\|}{\|x'-a\|}
	\le 1- \frac{\|x'-a\|-\|x-x'\|}{\|x'-a\|}
	= \frac{\|x-x'\|}{\|x'-a\|}.
	\end{align*}
	Hence \eqref{Step3} has been proved.
	
	A combination of \eqref{Step2} and \eqref{Step3} yields that
	\begin{align}\label{est_term1}
	\norm{\frac{x'-a}{\|x'-a\|}-\frac{x-a}{\|x-a\|}} \le 2\sqrt{\frac{2\de'-\de'^2}{4-6\de'+3\de'^2}}.
	\end{align}
	Plugging \eqref{est_term2} and \eqref{est_term1} into \eqref{key est_1} and using \eqref{del'2}, we obtain that
	\begin{align}\label{dis2N_A}
	\dist({x_1^*}',N_A(a)) \le \|{x_1^*}'-{x_1^*}\| &\le \norm{x_1^*}\paren{\sqrt{2\de'} + 2\sqrt{\frac{2\de'-\de'^2}{4-6\de'+3\de'^2}}}\\
	\nonumber
	&< \sqrt{2\de'} + 2\sqrt{\frac{2\de'-\de'^2}{4-6\de'+3\de'^2}}\;<\de.
	\end{align}
	The proof of $\dist({x_2^*}',N_B(b)) < \de$ is analogous and we also obtain that
	\begin{align}\label{dis2N_B}
	\dist({x_2^*}',N_B(b))\le \|{x_2^*}'-{x_2^*}\| &\le \norm{x_2^*}\paren{\sqrt{2\de'} + 2\sqrt{\frac{2\de'-\de'^2}{4-6\de'+3\de'^2}}}\\
	\nonumber
	&< \sqrt{2\de'} + 2\sqrt{\frac{2\de'-\de'^2}{4-6\de'+3\de'^2}}\;<\de.
	\end{align}
	Hence \eqref{satisfy 11} has been proved.
	
	Conditions \eqref{satisfy 11} and \eqref{satisfy 13} ensure that the pair of vectors $\{{x_1^*}',{x_2^*}'\}$ satisfies conditions \eqref{11} and \eqref{13}, respectively.
	It is trivial from the choice of $\de'$ at \eqref{del'2} that $a\in (A\setminus B)\cap \Ball_{\delta}(\bx)$, $b\in (B\setminus A)\cap \Ball_{\delta}(\bx)$.
We also have $x'\in \Ball_{\delta}(\bx)$ since
	\begin{align*}
\|x' -\bx\| 
&= \left \| x- \dfrac{\langle b-a, x- m \rangle}{\|b-a\|^2}(b-a) - \bar x
\right\| \le \|x -\bx\| + \|x - m\| \le \delta' + \max\{\|x-a\|,\|x-b\|\}\\
&\le \delta' + \|x-\bx\| + \max\{\|a-\bx\|, \|b-\bx\|\}
\le 3\delta' < \delta.
\end{align*}

	Hence, the estimate \eqref{>al} is applicable to $\{{x_1^*}',{x_2^*}'\}$. That is,
	\begin{equation}\label{est for x^*'}
	\|{x_1^*}'+{x_2^*}'\| > \al.
	\end{equation}
	Now using the triangle inequality, \eqref{dis2N_A}, \eqref{dis2N_B}, \eqref{est for x^*'}, \eqref{del'2} and \eqref{be} successively, we obtain the desired estimate:
	\begin{equation*}
	\begin{aligned}
	\|{x_1^*}+{x_2^*}\| &= \|{x_1^*}'+{x_2^*}'+{x_1^*}-{x_1^*}' + {x_2^*}-{x_2^*}'\|\\
	&\ge \|{x_1^*}'+{x_2^*}'\|-\|{x_1^*}'-{x_1^*}\|-\|{x_2^*}'-{x_2^*}\|\\
	&\ge \al - \paren{\norm{x_1^*}+\norm{x_2^*}}\paren{\sqrt{2\de'} + 2\sqrt{\frac{2\de'+\de'^2}{4+2\de'+\de'^2}}}\\
	&= \al - \paren{\sqrt{2\de'} + 2\sqrt{\frac{2\de'-\de'^2}{4-6\de'+3\de'^2}}}\\
	&> \al - (\al-\be) = \be.
	\end{aligned}
	\end{equation*}
	This completes \emph{Case 2} and \eqref{conclusion} has been proved.
	
	Hence, we have proved that $\{A,B\}$ is intrinsically transversal at $\bx$ with $\itr\ge\be$.
	Since $\be$ can be arbitrarily close to $\min\left\{\itrd{c},1/\sqrt{2}\right\}$, we also obtain the estimate \eqref{quan_est} and the proof is complete.
	\qed
\end{proof}

\begin{remark}\label{r:itrc<}
Technically, we are unable to get rid of the figure $1/\sqrt{2}$ in the key estimate \eqref{quan_est} of Theorem \ref{t:question 3}.
Fortunately, this technical burden does not restrict the analysis of the relevant transversality-type properties.
Let us briefly look at the case $\itrc > 1/\sqrt{2}$. 
Theorem \ref{t:question 3} together with Proposition \ref{p:relation_cons}$(i)$ implies that all the quantitative constants $\itr$, $\itrd{w}$, $\itrc$ and $\str$ are greater than $1/\sqrt{2}$ and, as a consequence, all the properties characterized by these constants are satisfied.
Indeed, the analysis of transversality-type properties would primarily address the question about the strict positiveness of the relevant characterizing constants rather than their boundedness from below by some positive number.
\end{remark}

Due to the technical burden as well as the ease for it discussed in Remark \ref{r:itrc<}, in the remainder of this section, we always make use of the assumption:
\begin{equation}\label{itrc<}
\itrc \le 1/\sqrt{2}.
\end{equation}

\begin{corollary}[equivalence of dual properties in Hilbert spaces]\label{c:equivalence}
	Let $\{A,B\}$ be a pair of closed sets and $\bx\in A\cap B$.
	Then it holds that
\begin{equation}\label{equal constant}
\itrc = \itr = \itrd{w}.
\end{equation}
\end{corollary}

\begin{proof}
A combination of Theorem \ref{t:question 3} and the inequality \eqref{itrc<} yields that $\itrc \le \itr$, which together with Proposition \ref{p:relation_cons}$(i)$ yields the equalities in \eqref{equal constant}.
\qed
\end{proof}

The next result significantly refines Proposition \ref{p:suff_subtran} in the Hilbert space setting, which is the weakest sufficient dual condition for subtransversality in the literature.

\begin{corollary}[refined sufficient dual condition for subtransversality]\label{c:suff_subtran}
A pair of closed sets $\{A,B\}$ is subtransversal at $\bx\in A\cap B$ if $\itrc>0$, that is, there exist numbers $\alpha\in]0,1[$ and $\delta>0$ such that $\|x^*_1+x^*_2\|>\alpha$
for all $a\in(A\setminus B)\cap \B_\de(\bx)$, $b\in(B\setminus A)\cap \B_\de(\bx)$, $x\in \B_\de(\bx)$ with $\|x-a\|=\|x-b\|$, and $x_1^*,x_2^*\in X^*$ satisfying 
\begin{gather*}
\dist(x_1^*,N_{A}(a))<\delta,\quad \dist(x_2^*,N_{B}(b))<\delta,\\
\|x^*_1\|+\|x^*_2\|=1,\quad
\langle x_1^*,x-a\rangle=\|x_1^*\|\,\|x-a\|,\;
\langle x_2^*,x-b\rangle=\|x_2^*\|\,\|x-b\|.
\end{gather*}
\end{corollary}

\begin{proof}
One the one hand, by Proposition \ref{p:relation_cons}$(ii)$, $\{A,B\}$ is subtransversal at $\bx$ if $\itrd{w}>0$.
On the other hand, it holds that $\itrc = \itrd{w}$ thanks to Corollary \ref{c:equivalence}.
Hence, the proof is complete.
\qed
\end{proof}

\begin{remark}[dual characterization of subtransversality with convexity]
	For a pair of closed and convex sets $\{A,B\}$, the property $\itrc > 0$ is not only sufficient but also necessary for subtransversality of $\{A,B\}$ at $\bx$.
\end{remark}

\begin{remark}\label{r:suff but not ness} As a by-product, we can now deduce the negative answer to the question whether the sufficient dual condition for subtransversality - Proposition \ref{p:suff_subtran} or equivalently Corollary \ref{c:suff_subtran} is also neccessary in the nonconvex setting.
Indeed, thanks to Corollary \ref{c:equivalence} the sufficient dual condition is equivalent to intrinsic transversality of $\{A,B\}$ at $\bar{x}$, which in turn implies local linear convergence of the alternating projection method around $\bx$ to a point in the intersection $A\cap B$ thanks to \cite[Theorem 6.1]{DruIofLew15}.
But in the nonconvex setting, there are pairs of closed sets that are subtransversal but the alternating projection method does not locally converge to a point in the intersection of the sets.
Hence, we infer that the dual condition $\itrc > 0$ (equivalently, $\itrd{w}>0$ or $\itr>0$) is not necesary for subtransversality.
\end{remark}

\begin{remark} We are ready to answer the research question raised by Kruger \cite[question 3, page 140]{Kru18} in the Hilbert space setting.
Recall that the property $\itrd{w} > 0$ is called \emph{weak intrinsic transversality} in \cite{KruLukNgu17,Kru18}.
The question is about the relationship between this property and intrinsic transversality in the general normed space setting.
The second equality of \eqref{equal constant} clearly shows that the two properties coincide in the Hilbert space setting.
Unfortunately, we have not obtained the answer in more general settings.
\end{remark}
 
In summary, Corollary \ref{c:equivalence} allows one to unify a number of dual transversality-type properties in the Hilbert space setting including intrinsic transversality, its weaker variant considered in \cite{Kru18}, the sufficient dual condition for subtransversality \cite{KruTha15,KruLukNgu18,KruLukNgu17} and the dual characterization of subtransversality with convexity \cite{Kru18}.
In our opinion, this significantly clarifies the picture of these important dual space properties.

\section{Intrinsic Transversality in Finite Dimensional Spaces}\label{s:finite dimensional}

We first recall definitions about relative limiting normals which are motivated by the compactness of the unit ball in finite dimensional spaces as well as the fact that not all normal vectors are always involved for characterizing transversality-type properties.
These notions were shown to be useful for analyzing the intrinsic transversality property and its variants, see \cite[page 123]{Kru18} for a more thorough discussion.

\begin{definition} \cite[Definition 2]{Kru18}\label{D2}
	Let $X$ be finite dimensional, $A,B\subset X$ and $\bx\in A\cap B$.
	\begin{enumerate}
		\item
		A pair $(x^*_1,x^*_2)\in X^*\times X^*$ is called
		a \emph{pair of relative limiting normals} to $\{A,B\}$ at $\bar x$ if there exist sequences $(a_k)\subset A\setminus B$, $(b_k)\subset B\setminus A$, $(x_k)\subset X$ and $(x_{1k}^*),(x_{2k}^*)\subset X^*$ such that $x_k\ne a_k$, $x_k\ne b_k$ $(k=1,2,\ldots)$, $a_k\to\bx$, $b_k\to\bx$, $x_k\to\bx$, $x_{1k}^*\to x_1^*$, $x_{2k}^*\to x_2^*$, and
		\begin{gather*}	
		x_{1k}^*\in N_A(a_k),\quad x_{2k}^*\in N_B(b_k)\quad (k=1,2,\ldots),
		\\
		\frac{\norm{x_k-a_k}}{\norm{x_k-b_k}} \to1,\quad
		\frac{\ang{x_{1k}^*,x_k-a_k}}{\norm{x_{1k}^*}\norm{x_k-a_k}} \to1,\quad
		\frac{\ang{x_{2k}^*,x_k-b_k}}{\norm{x_{2k}^*}\norm{x_k-b_k}} \to1,
		\end{gather*}
		with the convention that $\frac{0}{0}=1$.
		The collections of all pairs of relative limiting normals to $\{A,B\}$ at $\bar x$ will be denoted by $\overline{N}_{A,B}(\bar x)$.
		\item
		A pair $(x^*_1,x^*_2)\in X^*\times X^*$ is called
		a \emph{pair of restricted relative limiting normals} to $\{A,B\}$ at $\bar x$ if there exist sequences $(a_k)\subset A\setminus B$, $(b_k)\subset B\setminus A$, $(x_k)\subset X$ and $(x_{1k}^*),(x_{2k}^*)\subset X^*$ such that $\norm{x_k-a_k}=\norm{x_k-b_k}$ $(k=1,2,\ldots)$,
		$a_k\to\bx$, $b_k\to\bx$, $x_k\to\bx$, $x_{1k}^*\to x_1^*$, $x_{2k}^*\to x_2^*$, and
		\begin{gather*}
		\dist(x_{1k}^*,N_A(a_k))\to0,\quad \dist(x_{2k}^*,N_B(b_k))\to0,
		\\
		\ang{x_{1k}^*,x_k-a_k}=\norm{x_{1k}^*}\norm{x_k-a_k},\quad
		\ang{x_{2k}^*,x_k-b_k}=\norm{x_{2k}^*}\norm{x_k-b_k}\quad (k=1,2,\ldots).
		\end{gather*}
		The collections of all pairs of restricted relative limiting normals to $\{A,B\}$ at $\bar x$ will be denoted by $\overline{N}^c_{A,B}(\bar x)$.
	\end{enumerate}
\end{definition}

The following result shows that the two sets $\overline{N}_{A,B}(\bar x)$ and $\overline{N}^c_{A,B}(\bar x)$ are cones.

\begin{proposition} \cite[Proposition 2(i)]{Kru18}
		Let $A,B$ be subsets of a finite dimensional space $X$ and $\bx \in A \cap B$.
		The sets $\overline{N}_{A,B}(\bx)$ and $\overline{N}^c_{A,B}(\bx)$ are closed cones in $X^* \times X^*$, possibly empty.
		Moreover, if
		$(x^*_1,x^*_2) \in \overline{N}_{A,B}(\bx)$ or $(x^*_1,x^*_2) \in \overline{N}^c_{A,B}(\bx)$  then $(t_1x^*_1,t_2x^*_2) \in \overline{N}_{A,B}(\bx)$ or $(t_1x^*_1,t_2x^*_2) \in \overline{N}^c_{A,B}(\bx)$ for all $t_1,t_2 >0$.	
\end{proposition}

Thanks to the compactness of the sets under the following minima in finite dimensional spaces, making use of the notions in Definition \ref{D2} leads to alternative and more concise representations for $\itr$ and $\itrc$ compared to \eqref{itr constant} and \eqref{itrc constant}, respectively:
\begin{gather}\label{itr_constant_2}
\itr= \min_{\substack{(x^*_1,x^*_2)\in\overline{N}_{A,B}(\bar x)\\
		\|x^*_1\|+\|x^*_2\|=1}}\|x^*_1+x^*_2\|,
\\ \label{itrc_constant_2}
\itrc= \min_{\substack{(x^*_1,x^*_2)\in\overline{N}{}^c_{A,B}(\bar x)\\
		\|x^*_1\|+\|x^*_2\|=1}}\|x^*_1+x^*_2\|,
\end{gather}
with the convention that the minimum over the empty set equals $1$.

We have shown in Corollary \ref{c:equivalence} that the two constants are either both greater than $1/\sqrt{2}$ or equal in the Hilbert space setting.
How about the relationship between the two cones under the minima on the right-hand-side of \eqref{itr_constant_2} and \eqref{itrc_constant_2}?
It has only been known from \cite{Kru18} that $\overline{N}^c_{A,B}(\bar x) \subset \overline{N}_{A,B}(\bar x)$ since if $(x^*_1,x^*_2)\in X^*\times X^*$ is a \emph{pair of restricted relative limiting normals} to $\{A,B\}$ at $\bar x$, then it is also a \emph{pair of relative limiting normals} to $\{A,B\}$ at this point.
One of our goals in this section is to further investigate this relationship.
By doing this, we are indeed addressing the research question asked by Kruger \cite[question 4, page 140]{Kru18}.

To begin, let us define the constraint set below:
\begin{equation}\label{con_cone}
C := \{(x_1^*,x_2^*)\in X^* \times X^* \mid \ip{x_1^*}{x_2^*}\le 0\}.
\end{equation}

\begin{remark}
It is clear that $C$ is a closed cone in $X^* \times X^*$ and if $(x^*_1,x^*_2) \in C$, then $(t_1x^*_1,t_2x^*_2) \in C$ for all $t_1,t_2 \ge 0$.
\end{remark}

\begin{remark}\label{r:about C}
The set $C$ is consistently related to the assumption \eqref{itrc<} we imposed in Section \ref{s:Tran,subtran,itr tran}.
More specifically, condition \eqref{itrc<} holds true if and only if
	\begin{equation*}
C \cap \overline{N}^c_{A,B}(\bar x) \cap \{(x_1^*,x_2^*)\in X^* \times X^* \mid \|x_1^*\| + \|x_2^*\|=1\} \neq \emptyset.
	\end{equation*}
Under this condition, one can further refine the two expressions \eqref{itr_constant_2} and \eqref{itrc_constant_2}:
\begin{gather}\label{itr_constant_3}
\itr= \min_{\substack{(x^*_1,x^*_2)\in\overline{N}_{A,B}(\bar x)\\
		\|x^*_1\|+\|x^*_2\|=1}}\|x^*_1+x^*_2\| = \min_{\substack{(x^*_1,x^*_2)\in C \cap \overline{N}_{A,B}(\bar x)\\
		\|x^*_1\|+\|x^*_2\|=1}}\|x^*_1+x^*_2\|,
\\\label{itrc_constant_3}
\itrc= \min_{\substack{(x^*_1,x^*_2)\in\overline{N}{}^c_{A,B}(\bar x)\\
		\|x^*_1\|+\|x^*_2\|=1}}\|x^*_1+x^*_2\| = \min_{\substack{(x^*_1,x^*_2)\in C\cap \overline{N}{}^c_{A,B}(\bar x)\\
		\|x^*_1\|+\|x^*_2\|=1}}\|x^*_1+x^*_2\|,
\end{gather}
and, as a consequence, only pairs of vectors in $C$ are needed for calculating the two quantitative constants.
\end{remark}

In view of \eqref{itr_constant_3} and \eqref{itrc_constant_3} and Corollary \ref{c:equivalence}, a number of characterizations of intrinsic transversality can be recast in the next proposition, which slightly extends the list in \cite[Theorems 5 and 6]{Kru18}.

\begin{proposition}\label{p:characterizations of itr}
	Let $X$ be finite dimensional, $A,B\subset X$ be closed and $\bx\in A\cap B$. 
	The following conditions are equivalent:
	\begin{enumerate}
		\item
		$\{A,B\}$ is intrinsically transversal at $\bar x$;
		\item $\itr>0$;
		\item $\itrd{w}>0$;
		\item $\itrc>0$;
		\item
		there exists a number $\alpha\in]0,1[$ such that ${\|x^*_1+x^*_2\|>\alpha}$
		for all $(x^*_1,x^*_2)\in C\cap \overline{N}_{A,B}(\bar x)$ satisfying
		$\|x^*_1\|+\|x^*_2\|=1$;
		\item
		there exists a number $\alpha\in]0,1[$ such that ${\|x^*_1+x^*_2\|>\alpha}$
		for all $(x^*_1,x^*_2)\in C\cap \overline{N}^c_{A,B}(\bar x)$ satisfying
		$\|x^*_1\|+\|x^*_2\|=1$;
		\item
		$\left\{x^*\in X^*\mid (x^*,-x^*)\in\overline{N}_{A,B}(\bar x)\right\} \subset \{0\}$;
		\item
		$\left\{x^*\in X^*\mid (x^*,-x^*)\in\overline{N}^c_{A,B}(\bar x)\right\} \subset \{0\}$.
	\end{enumerate}
	If, in addition, the sets $A,B$ are convex, then the following item can be added to the above list.
	\begin{enumerate}
	\setcounter{enumi}{8}
	\item $\{A,B\}$ is subtransversal at $\bar x$.
	\end{enumerate}
\end{proposition}

\begin{remark} It is also possible to establish an expression for $\itrd{w}$ analogous to \eqref{itr_constant_3} and \eqref{itrc_constant_3} for $\itr$ and $\itrc$.
	Then two more characterizations of intrinsic transversality deduced from such an expression could be added to the list in Proposition \ref{p:characterizations of itr}.
	This task would lead to the introduction of another cone object containing $\overline{N}^c_{A,B}(\bar x)$ and being contained in $\overline{N}_{A,B}(\bar x)$.
	The analysis of such a cone object is indeed a research question asked by Kruger \cite[question 5, page 140]{Kru18}.
	However, we will show later on in Theorem \ref{t:question 4} that the latter two cones are indeed equal when restricted in the cone $C$ that is the case of our main interest (see Remark \ref{r:about C}).
	Hence, we choose not to give details about this task for simplicity in terms of presentation.  
\end{remark}

We now reveal a deeper relationship between $\lncone{A,B}(\bx)$ and $\lncone{A,B}^c(\bx)$.
This result complements Corollary \ref{c:equivalence} in the Euclidean space setting and further clarifies the characterization of intrinsic transversality in terms of (restricted) relative limiting normals.
Apart from the latter application, the next theorem was also inspired by the importance of the cones themselves, see \cite[page 123]{Kru18}. 

\begin{theorem}\label{t:question 4}
	Let $\{A,B\}$ be a pair of closed sets in a Euclidean space and $\bx\in A\cap B$.
	Then
	\begin{equation*}
	\lncone{A,B}^c(\bx) \cap C  =  \lncone{A,B}(\bx)  \cap C.
	\end{equation*}
\end{theorem}

\begin{proof}
	It is known by \cite[Proposition 2$(ii)$]{Kru18} that $\lncone{A,B}^c(\bx)  \subset  \lncone{A,B}(\bx)$.
	In this proof, we prove the inverse inclusion
	\begin{equation*}
	\lncone{A,B}(\bx)  \cap C  \subset  \lncone{A,B}^c(\bx)  \cap C.
	\end{equation*}
	Let us take any $(x_1^*,x_2^*) \in \lncone{A,B}(\bx)  \cap C$.
	Then by the definition of $\lncone{A,B}(\bx)$, there exist sequences $(a_k)\subset A\setminus B$, $(b_k)\subset B\setminus A$, $(x_k)\subset X$ and $(x_{1k}^*), (x_{2k}^*)\subset X^*$ such that $x_k\ne a_k$, $x_k\ne b_k$ $(k=1,2,\ldots)$, $a_k\to \bx$, $b_k\to \bx$, $x_k\to \bx$, $x_{1k}^*\to x_1^*$, $x_{2k}^*\to x_2^*$ and
	\begin{gather}\nonumber
	x_{1k}^*\in \ncone{A}(a_k),\quad x_{2k}^*\in \ncone{B}(b_k)\quad (k=1,2,\ldots),\\
	\label{2c_2}
	\frac{\|x_k-a_k\|}{\|x_k-b_k\|} \to 1,\quad \frac{\ip{x_{1k}^*}{x_k-a_k}}{\norm{x_{1k}^*}\|x_k-a_k\|} \to 1,\quad \frac{\ip{x_{2k}^*}{x_k-b_k}}{\norm{x_{1k}^*}\|x_k-b_k\|} \to 1.
	\end{gather}
	Since $(x_1^*,x_2^*) \in C$, it holds that $\ip{x_1^*}{x_2^*} \le 0$, equivalently,
	\begin{equation}\label{<0}
	\norm{\frac{x_1^*}{\norm{x_1^*}}-\frac{x_2^*}{\norm{x_2^*}}} \ge \sqrt{2}.
	\end{equation}
	To complete the proof, it suffices to prove that $(x_1^*,x_2^*) \in \lncone{A,B}^c(\bx)$.
	For each $k=1,2,\ldots$, let us define: 
	\begin{gather}\nonumber
	m_k := \frac{a_k+b_k}{2},\\
	\label{x'_k}
	x_k' := x_k - \frac{\ip{b_k-a_k}{x_k-m_k}}{\|b_k-a_k\|^2}(b_k-a_k),\\
	\label{x^*'_k}
	{x_{1k}^*}' := \frac{\|x_{1k}^*\|}{\|x_k'-a_k\|}(x_k'-a_k),\quad   {x_{2k}^*}' := \frac{\|x_{2k}^*\|}{\|x_k'-b_k\|}(x_k'-b_k).
	\end{gather}
	All we need is to verify the following four conditions:
	\begin{enumerate}
		\item\label{i}
		\begin{equation}\label{|x'-a|=|x'-b|_k}
		\|x_k'-a_k\|=\|x_k'-b_k\|\quad (k=1,2,\ldots);
		\end{equation}
		\item\label{ii}
		\begin{equation*}
		x_k'\to \bx,\quad {x_{1k}^*}'\to x_1^*,\quad {x_{2k}^*}'\to x_2^* \quad (k=1,2,\ldots);
		\end{equation*}		
		\item\label{iii}
		\begin{equation*}
		\dist\paren{{x_{1k}^*}',\ncone{A}(a_k)} \to 0,\quad \dist\paren{{x_{2k}^*}',\ncone{B}(b_k)} \to 0;
		\end{equation*}		
		\item\label{iv}
		\begin{equation*}
		\ip{{x_{1k}^*}'}{x_k'-a_k} = \|{x_{1k}^*}'\|\|x_k'-a_k\|,\quad
		\ip{{x_{2k}^*}'}{x_k'-b_k} = \|{x_{2k}^*}'\|\|x_k'-b_k\|.
		\end{equation*}
	\end{enumerate}
	
	\emph{Condition \eqref{i}:} this follows from \eqref{x'_k} since for each $k=1,2,\ldots$, we have that
	\begin{align*}
	\|x_k'-a_k\|^2-\|x_k'-b_k\|^2 &= \|x_k-a_k\|^2-\|x_k-b_k\|^2 - 2\frac{\ip{b_k-a_k}{x_k-m_k}}{\|b_k-a_k\|^2}\ip{x_k-a_k}{b_k-a_k}\\
	&+ 2\frac{\ip{b_k-a_k}{x_k-m_k}}{\|b_k-a_k\|^2}\ip{x_k-b_k}{b_k-a_k}\\
	&= \|x_k-a_k\|^2-\|x_k-b_k\|^2 - 2\frac{\ip{b_k-a_k}{x_k-m_k}}{\|b_k-a_k\|^2}\ip{b_k-a_k}{b_k-a_k}\\
	&= \|x_k-a_k\|^2-\|x_k-b_k\|^2 - 2\ip{b_k-a_k}{x_k-m_k}\\
	&= \|x_k-a_k\|^2-\|x_k-b_k\|^2 - \ip{(x_k-a_k)-(x_k-b_k)}{(x_k-a_k) + (x_k-b_k)}=0.
	\end{align*}
	
	\emph{Condition \eqref{iv}:} from \eqref{x^*'_k}, we have that
	\begin{gather*}
	\|{x_{1k}^*}'\| = \|x_{1k}^*\|,\; \|{x_{2k}^*}'\| = \|x_{2k}^*\|,\\
	\ip{{x_{1k}^*}'}{x_k'-a_k} = \|{x_{1k}^*}\|\|x_k'-a_k\| = \|{x_{1k}^*}'\|\|x_k'-a_k\|,\\
	\ip{{x_{2k}^*}'}{x_k'-b_k} = \|{x_{2k}^*}\|\|x_k'-b_k\|= \|{x_{2k}^*}'\|\|x_k'-b_k\|.
	\end{gather*}
	
	\emph{Condition \eqref{iii}:} since $x_{1k}^*\in \ncone{A}(a_k)$ and $x_{2k}^*\in \ncone{A}(b_k)$ $(k=1,2,\ldots)$, whenever condition \eqref{ii} has been verified, we have that
	\begin{gather*}
	\dist\paren{{x_{1k}^*}',\ncone{A}(a_k)} \le \norm{{x_{1k}^*}'-x_{1k}^*} \le \norm{{x_{1k}^*}'-x_1^*}+ \norm{{x_{1k}^*}-x_1^*} \to 0,\\
	\dist\paren{{x_{2k}^*}',\ncone{B}(b_k)} \le \norm{{x_{2k}^*}'-x_{2k}^*} \le \norm{{x_{2k}^*}'-x_2^*}+ \norm{{x_{2k}^*}-x_2^*} \to 0.
	\end{gather*}
	
	\emph{Condition \eqref{ii}:} since $x_k\to \bx$, $a_k\to \bx$ and $b_k\to \bx$, it holds by \eqref{x'_k} that
	\begin{align*}
	\norm{x_k' - x_k} = \norm{\frac{\ip{b_k-a_k}{x_k-m_k}}{\|b_k-a_k\|^2}(b_k-a_k)} \le \norm{x_k-m_k}
	=  \norm{x_k-\frac{a_k+b_k}{2}} \to 0.
	\end{align*}
	Then
	\begin{align*}
	\norm{x_k' - \bx} \le \norm{x_k' - x_k} + \norm{x_k - \bx} \to 0.
	\end{align*}
	In the remainder of the proof, we show that ${x_{1k}^*}'\to x_1^*$ while the condition ${x_{2k}^*}'\to x_2^*$ is obtained in a similar manner. Since ${x_{1k}^*}\to x_1^*$, we need to show that $\norm{{x_{1k}^*}'-{x_{1k}^*}}\to 0$.
	Note that by \eqref{x^*'_k} it holds that
	\begin{equation}\label{key est_1_k}
	\begin{aligned}
	\|{x_{1k}^*}'-x_{1k}^*\| &= \norm{\frac{\|x_{1k}^*\|}{\|x_k'-a_k\|}(x_k'-a_k)-x_{1k}^*} = \|x_{1k}^*\|\norm{\frac{x_k'-a_k}{\|x_k'-a_k\|}-\frac{x_{1k}^*}{\|x_{1k}^*\|}}\\
	&\le \|x_{1k}^*\|\paren{\norm{\frac{x_k-a_k}{\|x_k-a_k\|}-\frac{x_{1k}^*}{\|x_{1k}^*\|}}+
		\norm{\frac{x_k'-a_k}{\|x_k'-a_k\|}-\frac{x_k-a_k}{\|x_k-a_k\|}}}.
	\end{aligned}
	\end{equation}
	Note also that due to \eqref{2c_2},
	\begin{equation}\label{key est_2_k}
	\norm{\frac{x_k-a_k}{\|x_k-a_k\|}-\frac{x_{1k}^*}{\|x_{1k}^*\|}} = \sqrt{2 - 2\frac{\ip{x_{1k}^*}{x_k-a_k}}{\norm{x_{1k}^*}\|x_k-a_k\|}}\; \to 0.
	\end{equation}
	In view of \eqref{key est_1_k} and \eqref{key est_2_k}, to obtain $\norm{{x_{1k}^*}'-{x_{1k}^*}}\to 0$, it suffices to prove that
	\begin{align*}
	\norm{\frac{x_k'-a_k}{\|x_k'-a_k\|}-\frac{x_k-a_k}{\|x_k-a_k\|}} \to 0.
	\end{align*}
	To proceed, let us take any number $\eps>0$ which can be arbitrarily small and show the existence of an natural $N\in \mathbb{N}$ such that
	\begin{equation}\label{main_est_q4}
	\norm{\frac{x_k'-a_k}{\|x_k'-a_k\|}-\frac{x_k-a_k}{\|x_k-a_k\|}} < \eps, \quad \forall k\ge N.
	\end{equation}
	Choose a number $\eps'>0$ and satisfying
	\begin{align}\label{eps'}
	2\sqrt{\frac{2\eps'-\eps'^2}{4-6\eps'+3\eps'^2}}\;<\eps.
	\end{align}
	Such a number $\eps'$ exists since $\eps>0$ and
$\lim_{t\downarrow 0}\; 2\sqrt{\frac{2t-t^2}{4-6t+3t^2}}\;=0$.\\
	By the convergence conditions in \eqref{2c_2}, there exists a natural number $N\in \mathbb{N}$ such that $\forall k\ge N$,
	\begin{gather}\label{2c_2_eps}
	1-\eps' < \frac{\|x_k-a_k\|}{\|x_k-b_k\|} < 1+\eps',\\
	\label{2c_2_eps1}
	\frac{\ip{x_{1k}^*}{x_k-a_k}}{\norm{x_{1k}^*}\|x_k-a_k\|} > 1-\eps',\quad \frac{\ip{x_{2k}^*}{x_k-b_k}}{\norm{x_{1k}^*}\|x_k-b_k\|} > 1-\eps'.
	\end{gather}
	The estimates in \eqref{2c_2_eps1} amount to
	\begin{gather}\label{2c_2_eps2}
	\norm{\frac{x_{1k}^*}{\norm{x_{1k}^*}}-\frac{x_k-a_k}{\norm{x_k-a_k}}} < \sqrt{2\eps'},\quad \norm{\frac{x_{2k}^*}{\norm{x_{2k}^*}}-\frac{x_k-b_k}{\norm{x_k-b_k}}} < \sqrt{2\eps'}.
	\end{gather}

	In order to prove \eqref{main_est_q4}, we first note that
	\begin{equation}\label{ip=0_k}
	\ip{x_k-x_k'}{x_k'-m_k}=0.
	\end{equation}
	Indeed, by \eqref{x'_k}, it holds that
	\begin{equation*}
	\ip{x_k-x_k'}{x_k'-m_k} = \frac{\ip{b_k-a_k}{x_k-m_k}}{\|b_k-a_k\|^2}\ip{b_k-a_k}{x_k'-m_k},
	\end{equation*}
	from which \eqref{ip=0_k} follows since
	\begin{align*}
	\ip{b_k-a_k}{x_k'-m_k} &= \ip{b_k-a_k}{x_k-m_k-\frac{\ip{b_k-a_k}{x_k-m_k}}{\|b_k-a_k\|^2}(b_k-a_k)}\\
	&= \ip{b_k-a_k}{x_k-m_k}-\frac{\ip{b_k-a_k}{x_k-m_k}}{\|b_k-a_k\|^2}\ip{b_k-a_k}{b_k-a_k}=0.
	\end{align*}
	
	Second, we show that $\forall k\ge N$,
	\begin{equation}\label{Step1_k}
	\|x_k-x_k'\|^2 \le \frac{2\eps'-\eps'^2}{4(1-\eps')^2}\min\left\{\|x_k-a_k\|^2,\|x_k-b_k\|^2\right\}.
	\end{equation}

If $\|x_k-a_k\|\ge\|x_k-b_k\|$, then
	\begin{gather}\label{<x-m,b-a> >0_k}
\ip{b_k-a_k}{x_k-m_k} \ge 0.
	\end{gather}
	Note from \eqref{<0} and \eqref{2c_2_eps2} that
	\begin{align*}
	\norm{\frac{x_k-a_k}{\norm{x_k-a_k}}-\frac{x_k-b_k}{\norm{x_k-b_k}}} &\ge \norm{\frac{x_{1k}^*}{\norm{x_{1k}^*}}-\frac{x_{2k}^*}{\norm{x_{2k}^*}}} - \norm{\frac{x_{1k}^*}{\norm{x_{1k}^*}}-\frac{x_k-a_k}{\norm{x_k-a_k}}} - \norm{\frac{x_{2k}^*}{\norm{x_{2k}^*}}-\frac{x_k-b_k}{\norm{x_k-b_k}}}\\
	&> \sqrt{2} - 2\sqrt{2\eps'}.
	\end{align*}
	This combining with \eqref{2c_2_eps} yields that
	\begin{equation*}
	\ip{x_k-a_k}{x_k-b_k} < 4\paren{\sqrt{\eps'}-\eps'}\|x_k-a_k\|\|x_k-b_k\| < 4\paren{\sqrt{\eps'}-\eps'}(1+\eps')\|x_k-b_k\|^2.
	\end{equation*}
	Then
	\begin{equation}\label{<x-b,b-a> >0_k}
	\begin{aligned}
	\ip{b_k-a_k}{x_k-b_k} &= \ip{b_k-x_k}{x_k-b_k} + \ip{x_k-a_k}{x_k-b_k}\\
	&= -\norm{x_k-b_k}^2 + \ip{x_k-a_k}{x_k-b_k}\\
	&< -\norm{x_k-b_k}^2 + 4\paren{\sqrt{\eps'}-\eps'}(1+\eps')\|x_k-b_k\|^2\\
	&= -\paren{1-4\paren{\sqrt{\eps'}-\eps'}(1+\eps')}\|x_k-b_k\|^2<0.
	\end{aligned}
	\end{equation}
	From \eqref{x'_k}, \eqref{<x-m,b-a> >0_k} and \eqref{<x-b,b-a> >0_k} we have that
	\begin{equation}\label{|x'-b|>|x-b|_k}
	\begin{aligned}
	\|x_k'-b_k\|^2 - \|x_k-b_k\|^2 &= \|x_k'-x_k\|^2+2\ip{x_k'-x_k}{x_k-b_k}\\
	&= \|x_k'-x_k\|^2 - 2\frac{\ip{b_k-a_k}{x_k-m_k}}{\|b_k-a_k\|^2}\ip{b_k-a_k}{x_k-b_k}\\
	&\ge \|x_k'-x_k\|^2.
	\end{aligned}
	\end{equation}
	By \eqref{|x'-a|=|x'-b|_k} and \eqref{ip=0_k} we get that
	\begin{align*}
	\|x_k-a_k\|^2+\|x_k-b_k\|^2 &= 2\|x_k-x_k'\|^2 + \|x_k'-a_k\|^2 + \|x_k'-b_k\|^2 + 2\ip{x_k-x_k'}{2x_k'-(a_k+b_k)} \\
	&= 2\|x_k-x_k'\|^2 + 2\|x_k'-b_k\|^2 + 4\ip{x_k-x_k'}{x_k'-m_k}\\
	&= 2\|x_k-x_k'\|^2 + 2\|x_k'-b_k\|^2.
	\end{align*}
	This together with \eqref{2c_2_eps} and \eqref{|x'-b|>|x-b|_k} yields
	\begin{align*}
	2\|x_k-x_k'\|^2 &= \|x_k-a_k\|^2+\|x_k-b_k\|^2 - 2\|x_k'-b_k\|^2\\
	&\le (1+\eps')^2\|x_k-b_k\|^2 +\|x_k-b_k\|^2 - 2\|x_k'-b_k\|^2\\
	&\le (1+\eps')^2\|x_k-b_k\|^2 +\|x_k-b_k\|^2 - 2\paren{\|x_k-b_k\|^2+\|x_k'-x_k\|^2}.
	\end{align*}
	Hence
	\begin{equation}\label{Step1_case1_k}
	4\|x_k-x_k'\|^2 \le \paren{2\eps'+\eps'^2}\|x_k-b_k\|^2 = \paren{2\eps'+\eps'^2}\min\{\|x_k-a_k\|^2,\|x_k-b_k\|^2\}
	\end{equation}
	since $\|x_k-a_k\|\ge \|x_k-b_k\|$ in this case.

By a similar argument, if $\|x_k-a_k\|\le \|x_k-b_k\|$, then
\[
\ip{b_k-a_k}{x_k-m_k} \le 0,\;
\ip{b_k-a_k}{x_k-a_k} \ge 0.
\]
Thus
	\begin{equation}\label{|x'-a|>|x-a|_k}
	\begin{aligned}
\|x_k'-a_k\|^2 - \|x_k-a_k\|^2 &= \|x_k'-x_k\|^2+2\ip{x_k'-x_k}{x_k-a_k}\\
	&= \|x_k'-x_k\|^2 - 2\frac{\ip{b_k-a_k}{x_k-m_k}}{\|b_k-a_k\|^2}\ip{b_k-a_k}{x_k-a_k}\\
	&\ge \|x_k'-x_k\|^2.
	\end{aligned}
	\end{equation}
By \eqref{|x'-a|=|x'-b|_k} and \eqref{ip=0_k} we get that
	\begin{align*}
	\|x_k-a_k\|^2+\|x_k-b_k\|^2 = 2\|x_k-x_k'\|^2 + 2\|x_k'-a_k\|^2,
	\end{align*}
which together with \eqref{2c_2_eps} and \eqref{|x'-a|>|x-a|_k} yields that
	\begin{align*}
	2\|x_k-x_k'\|^2 \le \|x_k-a_k\|^2 + \frac{1}{(1-\eps')^2}\|x_k-a_k\|^2 - 2\paren{\|x_k-a_k\|^2+\|x_k'-x_k\|^2}.
	\end{align*}
Equivalently,
	\begin{equation}\label{Step1_case2_k}
	4\|x_k-x_k'\|^2 \le \frac{2\eps'-{\eps'}^2}{(1-\eps')^2}\|x_k-a_k\|^2 = \frac{2\eps'-\eps'^2}{(1-\eps')^2}\min\{\|x_k-a_k\|^2,\|x_k-b_k\|^2\}
	\end{equation}
	since $\|x_k-a_k\|\le \|x_k-b_k\|$ in this case.

Combining \eqref{Step1_case1_k} and \eqref{Step1_case2_k} and noting that $2\eps'+\eps'^2 < \frac{2\eps'-\eps'^2}{(1-\eps')^2}$, we obtain \eqref{Step1_k} as claimed.
	
	Third, we show that $\forall k\ge N$
	\begin{align}\label{Step2_k}
	\|x_k'-a_k\|^2 \ge \frac{4-6\eps'+3\eps'^2}{2\eps'-\eps'^2}\|x_k-x_k'\|^2.
	\end{align}
Indeed, if $\|x_k-a_k\|\le \|x_k-b_k\|$, then the use of \eqref{|x'-a|>|x-a|_k} and \eqref{Step1_case2_k} yields \eqref{Step2_k}:
	\begin{align*}
	 \|x_k'-a_k\|^2 &\ge \|x_k-a_k\|^2+\|x_k-x_k'\|^2\\
	&\ge \frac{4(1-\eps')^2)}{2\eps'-\eps'^2}\|x_k-x_k'\|^2+\|x_k-x_k'\|^2 = \frac{4-6\eps'+3\eps'^2}{2\eps'-\eps'^2}\|x_k-x_k'\|^2.
	\end{align*}
		Otherwise, i.e., $\|x_k-a_k\|\ge \|x_k-b_k\|$, then the use of \eqref{|x'-a|=|x'-b|_k}, \eqref{|x'-b|>|x-b|_k} and \eqref{Step1_k} successively implies that
	\begin{align*}
	&\|x_k'-a_k\|^2 = \|x_k'-b_k\|^2 \ge \|x_k-b_k\|^2+\|x_k-x_k'\|^2\\
	\ge &\frac{4}{2\eps'+\eps'^2}\|x_k-x_k'\|^2+\|x_k'-x_k\|^2 = \frac{4+2\eps'+\eps'^2}{2\eps'+\eps'^2}\|x_k-x_k'\|^2,
	\end{align*}
which also yields \eqref{Step2_k} since $\frac{4+2\eps'+\eps'^2}{2\eps'+\eps'^2} > \frac{4-6\eps'+3\eps'^2}{2\eps'-\eps'^2}$.
Hence \eqref{Step2_k} has been proved.
	
	Fourth, we show that $\forall k\ge N$
	\begin{align}\label{Step3_k}
	\norm{\frac{x_k'-a_k}{\|x_k'-a_k\|}-\frac{x_k-a_k}{\|x_k-a_k\|}} \le 2\frac{\|x_k-x_k'\|}{\|x_k'-a_k\|}.
	\end{align}
	Indeed,
	\begin{align*}
	\norm{\frac{x_k'-a_k}{\|x_k'-a_k\|}-\frac{x_k-a_k}{\|x_k-a_k\|}} &\le \norm{\frac{x_k'-a_k}{\|x_k'-a_k\|}-\frac{x_k-a_k}{\|x_k'-a_k\|}} + \norm{\frac{x_k-a_k}{\|x_k'-a_k\|}-\frac{x_k-a_k}{\|x_k-a_k\|}}\\
	&= \frac{\|x_k-x_k'\|}{\|x_k'-a_k\|} + \left|\frac{\|x_k-a_k\|}{\|x_k'-a_k\|}-1\right|.
	\end{align*}
	If $\|x_k-a_k\|\ge \|x_k'-a_k\|$, then \eqref{Step3_k} holds true since
	\begin{align*}
	\left|\frac{\|x_k-a_k\|}{\|x_k'-a_k\|}-1\right| = \frac{\|x_k-a_k\|}{\|x_k'-a_k\|}-1\le  \frac{\|x_k-x_k'\|+\|x_k'-a_k\|}{\|x_k'-a_k\|}-1
	=\frac{\|x_k-x_k'\|}{\|x_k'-a_k\|}.
	\end{align*}
	Otherwise, i.e., $\|x_k-a_k\|<\|x_k'-a_k\|$, \eqref{Step3_k} also holds true since
	\begin{align*}
	\left|\frac{\|x_k-a_k\|}{\|x_k'-a_k\|}-1\right| = 1- \frac{\|x_k-a_k\|}{\|x_k'-a_k\|}
	\le 1- \frac{\|x_k'-a_k\|-\|x_k-x_k'\|}{\|x_k'-a_k\|}
	= \frac{\|x_k-x_k'\|}{\|x_k'-a_k\|}.
	\end{align*}
	Hence \eqref{Step3_k} has been proved.
	
	Finally, a combination of \eqref{Step2_k}, \eqref{Step3_k} and \eqref{eps'} yields that
	\begin{align*}
	\norm{\frac{x_k'-a_k}{\|x_k'-a_k\|}-\frac{x_k-a_k}{\|x_k-a_k\|}}\le 2\frac{\|x_k-x_k'\|}{\|x_k'-a_k\|} \le 2\sqrt{\frac{2\eps'-\eps'^2}{4-6\eps'+3\eps'^2}}\;<\eps,\quad \forall k\ge N,
	\end{align*}
	which is \eqref{main_est_q4} and hence the proof is complete.
	\qed
\end{proof}

\begin{remark}
In the Euclidean space setting, thanks to \eqref{itr_constant_3} and \eqref{itrc_constant_3}, Corollary \ref{c:equivalence} can easily be deduced from Theorem \ref{t:question 4}.
But the inverse implication is not trivial since the minimal values at \eqref{itr_constant_3} and \eqref{itrc_constant_3} being equal does not tell much about the relationship between the two feasibility sets there. 	
\end{remark}

\section{Primal Space Characterizations of Intrinsic Transversality}\label{s:primal}

In the Hilbert space setting, it can be deduced from Proposition \ref{p:relation_cons}(iii) and Corollary \ref{c:equivalence} that for pairs of closed and convex sets, intrinsic transversality is equivalent to subtransversality which is a primal space property.
The situation for pairs of nonconvex sets has not been known and there is an interest to research for primal space counterparts of intrinsic transversality in this setting.
This research question was raised by Ioffe \cite[page 358]{Iof17}.
Our agenda in this section is to present material sufficient for formulating a primal characterization of intrinsic transversality in the Hilbert space setting.

In the sequel, we always assume that the Cartesian product space $X\times X$ is endowed with the maximum norm and accordingly define the distance between two subsets of as follows: for any $P,Q \subset X\times X$,
\begin{equation}\label{dist in X^2}
\begin{aligned}
\dist\paren{P,Q} :=&\; \inf_{(p_1,p_2)\in P,(q_1,q_2)\in Q}\; \norm{(p_1,p_2)-(q_1,q_2)}_{\text{max}}\\
=&\; \inf_{(p_1,p_2)\in P,(q_1,q_2)\in Q}\; \max\{\norm{p_1-q_1},\norm{p_2-q_2}\}
\end{aligned}
\end{equation}
with the convention that the infimum over the empty set equals infinity.
For convenience, for a subset $\Omega\subset X$, we use the following notation:
\begin{equation}\label{[]_2}
[\Omega]_2:=\{(x,x) \in X\times X \mid x\in \Omega\}.
\end{equation}
Note that $[\Omega]_2$ is different from (smaller than) the Cartesian product set $\Omega\times\Omega$.
We frequently use the distance \eqref{dist in X^2} involving a Cartesian product set and a set of form \eqref{[]_2}:
$$
\dist\paren{A\times B,\parenn{\Omega}_2} = \inf_{x\in\Omega,a\in A,b\in B}\; \max\{\norm{x-a},\norm{x-b}\},\quad A,B,\Omega \subset X.
$$

We formulate several technical results which are essential for proving the key estimates in this section.

\begin{proposition}\label{C3} \cite[Corollary 6.3]{BuiKru} \footnote{The result is valid in Banach spaces.}
	Let $\left\{\widetilde A, \widetilde B \right\}$ be a pair of closed sets in $X$, $\bx \in \widetilde A\cap \widetilde B$, $u,v\in X$ and numbers $\rho,\eps >0$.
	Suppose that
	\begin{gather}\label{T3.1.1}
	(\widetilde A-u)\cap (\widetilde B-v)\cap \B_\rho(\bx)=\emptyset,\\
	\label{T3.1.1_Thao}
\max\{\norm{u},\norm{v}\}<\dist\paren{(\widetilde A-u)\times (\widetilde B-v),\parenn{\B_\rho(\bx)}_2}+\eps.
	\end{gather}
	Then, for any numbers $\la\ge \eps+\rho$ and $\tau \in \left]0,\frac{\la-\eps}{\la+\eps}\right[$ there exist points $\tilde a\in \widetilde A\cap \B_\la(\bx)$, $\tilde b\in \widetilde B\cap \B_\la(\bx)$, $\tilde x\in \B_\rho(\bx)$ and vectors $x^*_1,x^*_2\in X^*$ such that
	\begin{gather}
	\label{T3.1.2}
	\norm{x^*_1}+\norm{x_2^*}=1,\quad \norm{x_1^*+x_2^*}<\frac{\eps}{\rho},\\
	\label{T3.1.2.1}
	x_1^*\in N_{\widetilde A}(\tilde a),\quad x_2^*\in N_{\widetilde B}(\tilde b),\\
	\label{T3.1.3}
	\langle x_1^*,\tilde x-\tilde a+u\rangle +\langle x_2^*,\tilde x-\tilde b+v\rangle>\tau \max\{\norm{\tilde x-\tilde a+u},\norm{\tilde x-\tilde b+v}\}.
	\end{gather}
\end{proposition}

Condition \eqref{T3.1.3} plays an important role in our analysis. It relates the dual space elements $x_1^*,x_2^*$ to the primal space ones $\tilde x-\tilde a+u$, $\tilde x-\tilde b+v$.

The next lemma slightly modifies Proposition \ref{C3} in such a way that the imposed condition about the common point of the sets can be relaxed.

\begin{lemma}\label{C3.2}
Let $\{A,B\}$ be a pair of closed sets in $X$, $x\in X$, $a\in A$, $b\in B$ and numbers $\rho,\eps >0$.
Suppose that
\begin{equation}\label{l1 assume}
\eps < \max\{\norm{x-a},\norm{x-b}\} < \dist\paren{A\times B,\parenn{\B_\rho(x)}_2} + \eps.
\end{equation}
Then, for any numbers $\la\ge \eps+\rho$ and $\tau \in \left]0,\frac{\la-\eps}{\la+\eps}\right[$ there exist points $\hat a\in A\cap \B_\la(a)$, $\hat b\in B\cap \B_\la(b)$, $\hat x\in \B_\rho(x)$ and vectors $x^*_1,x^*_2\in X^*$ such that condition \eqref{T3.1.2} is satisfied and
\begin{gather}
\label{C3.1.2.1}
x_1^*\in N_A(\hat a),\quad x_2^*\in N_B(\hat b),\\
\label{C3.1.3}
\langle x_1^*,\hat x -\hat a\rangle +\langle x_2^*,\hat x-\hat b\rangle>\tau \max\left\{\norm{\hat x-\hat a},\|\hat x-\hat b\|\right\}.
\end{gather}
\end{lemma}

\begin{proof}
We are going to apply Proposition \ref{C3} for
\begin{gather}\label{til AB}
\widetilde A := A-a,\quad \widetilde B := B-b,\quad \bx := 0 \in \widetilde A\cap \widetilde B,\quad u:=x-a,\quad v:=x-b
\end{gather} 
 and the same technical numbers $\rho,\eps,\la,\tau$ by verifying conditions \eqref{T3.1.1} and \eqref{T3.1.1_Thao}.

Indeed, assumption \eqref{l1 assume} implies that $\dist\paren{A\times B,\parenn{\B_\rho(x)}_2}>0$ which
particularly implies that
$$
A\cap B\cap \B_\rho(x)=\emptyset,\; \text{equivalently,}\; (A-x)\cap (B-x)\cap (\rho\B)=\emptyset.
$$
This together with 
\[
(\widetilde A-u) \cap (\widetilde B-v) \cap (\B_\rho(\bx)) =(A-x)\cap (B-x)\cap (\rho\B)
\]
clearly yields \eqref{T3.1.1}.
Note also that
\begin{align*}
\dist\paren{(\widetilde A-u)\times (\widetilde B-v),\parenn{\B_\rho(\bx)}_2}
= \dist\paren{(A-x)\times (B-x),\parenn{\rho\B}_2}
= \dist\paren{A\times B,\parenn{\B_\rho(x)}_2},
\end{align*}
which means that condition \eqref{T3.1.1_Thao} is exactly the assumption \eqref{l1 assume}.

Then Proposition~\ref{C3} yields that there are points $\tilde a\in \widetilde A\cap (\la\B)$, $\tilde b\in \widetilde B\cap (\la\B)$, $\tilde x\in \rho\B$ and vectors $x^*_1,x^*_2\in X^*$ satisfying conditions \eqref{T3.1.2}, \eqref{T3.1.2.1} and \eqref{T3.1.3}.

Define $\hat a := \tilde{a} + a$, $\hat{b} :=\tilde{b}+b$ and $\hat{x} :=\tilde{x}+x$.
Then by the construction of $\widetilde A$, $\widetilde B$ at \eqref{til AB}, we have that $\hat a\in A\cap \B_\la(a)$, $\hat b\in B\cap \B_\la(b)$, $\hat x\in \B_\rho(x)$ and also
\begin{gather}\label{g1}
N_A(\hat a) = N_{\widetilde A}(\tilde{a}),\quad N_B(\hat b)=N_{\widetilde B}(\tilde{b}),
\\\label{g2}
\hat{x} - \hat{a} = \tilde{x}-\tilde{a}+x-a = \tilde{x}-\tilde{a}+u,\quad \hat{x} - \hat{b} = \tilde{x}-\tilde{b}+x-b = \tilde{x}-\tilde{b}+v.
\end{gather}
The combination of \eqref{T3.1.2.1} with \eqref{g1} yields \eqref{C3.1.2.1} and the combination of \eqref{T3.1.3} with \eqref{g2} yields \eqref{C3.1.3}.

That is, the points $\hat{a}$, $\hat{b}$, $\hat{x}$ and vectors $x_1^*,x_2^*$ satisfy all the required conditions of the lemma and hence the proof is complete.
\qed
\end{proof}

\begin{lemma}\label{P3.2}\footnote{The result is valid in normed linear spaces.}
	Let $\{A,B\}$ be a pair of closed sets in $X$, $a\in A$, $b\in B$, $x\in X$ with $\norm{x-a}=\norm{x-b}>0$, $\al>0$, $\eps\ge 0$ and vectors $x_1^*,x_2^*\in{X}^*$ with $\norm{x_1^*}+\norm{x_2^*}=1$.
	Suppose that the following conditions are satitisfied:
	\begin{gather}\label{Thao:sum=1}
{2\eps}({\al+1/\sqrt{\eps}})<\al,
\\\label{P3.2.1.1}
	{\dist(x_1^*,N_{A}(a)) <\eps},\quad \dist(x_2^*,N_{B}(b))<\eps,\quad\norm{x_1^*+x_2^*}+{2\eps}({\al+1/\sqrt{\eps}})<\al,\\
	\label{P3.2.}
	\langle x_1^*,x-a\rangle=\norm{x_1^*}\norm{x-a},\quad \langle x_2^*,x-b\rangle=\norm{x_2^*}\norm{x-b}.
	\end{gather}
Then there exists a number $\de>0$ such that for any $\rho\in\,]0,\de\,[$, we have that
	\begin{equation}\label{P3.2.3}
	\dist\paren{\paren{A\cap \overline\B_\la(a)} \times \paren{B\cap \overline\B_\la(b)},\parenn{\B_\rho(x)}_2} + \al\rho > \norm{x-a},
	\end{equation}
where $\la:=(\al+1/\sqrt{\eps})\rho$.
\end{lemma}

\begin{proof}
By \eqref{P3.2.1.1}, there exist vectors $a^*\in N_{A}(a)$, $b^*\in N_B(b)$ such that
	\begin{equation}\label{P5.13P1}
	\norm{x_1^*-a^*}\le\eps,\quad \norm{x_2^*-b^*} \le \eps.
	\end{equation}
Since $\al-\norm{x_1^*+x_2^*}>{2\eps}(\al+1/\sqrt{\eps})$, there are positive numbers $\al_1,\al_2$ such that 
\begin{equation}\label{P5.13P1.1}
	\al>\al_1>\al_2>\norm{x_1^*+x_2^*}, \quad 	\al_1-\al_2> {2\eps}({\al+1/\sqrt{\eps}}).
\end{equation} 
Choose a number $\be>0$ such that
\begin{equation}\label{al-al3}
\be<\frac{\al_1-\al_2}{\al+1/\sqrt{\eps}}-2\eps,\; \text{equivalently},\;
\al_1-\al_2-({\be+2\eps})({\al+1/\sqrt{\eps}})>0.
\end{equation}
	By the definition \eqref{NC} of the Fr\'echet normal cone, there is a number $\de_1>0$ such that
	\begin{equation}\label{P5.13P2}
	\langle a^*,a'-a\rangle \le\frac{\be}{2}\norm{a'-a},\; \langle b^*,b'-b\rangle \le\frac{\be}{2}\norm{b'-b},\quad
	\forall	a'\in A\cap \B_{\de_1}(a), b'\in B\cap \B_{\de_1}(b).
	\end{equation}
Let us define $\de:=\frac{\de_1}{\al+1/\sqrt{\eps}}>0$ and show that $\de$ fulfils the requirement of Lemma \ref{P3.2}.
	
	Indeed, let us suppose to the contrary that condition \eqref{P3.2.3} is not satisfied, i.e., there is a $\rho \in ]0,\delta[$ such that
\begin{equation*}
	\dist\paren{\paren{A\cap \overline\B_\la(a)} \times \paren{B\cap \overline\B_\la(b)},\parenn{\B_\rho(x)}_2}+\rho\al \le \norm{x-a},
\end{equation*}
where $\la:=(\al+1/\sqrt{\eps})\rho$.
Since $\al_1<\al$, the previous inequality ensures the existence of $x'\in \B_\rho(x)$, $a'\in A\cap \overline \B_\la(a)$ and $b'\in B\cap \overline\B_\la(b)$ such that
	\begin{equation}\label{P5.3.2}
		\max\{\norm{x'-a'},\norm{x'-b'}\}< \norm{x-a}-\rho\al_1.
	\end{equation}

We now make several observations as below.
First, by \eqref{P3.2.}, $\|x_1^*\|+\|x_2^*\|=1$ and $\|x-a\|=\|x-b\|$, it holds that
	\begin{equation}\label{l2:Thao_est3}
	\langle x_1^*,x-a\rangle+\langle x_2^*,x-b\rangle = \|x_1^*\|\|x-a\| + \|x_2^*\|\|x-b\| = \norm{x-a}.
\end{equation}

Second, by \eqref{P5.3.2} and $\|x_1^*\|+\|x_2^*\|=1$, it holds that
	\begin{equation}\label{l2:Thao_est4}
	\langle x_1^*,x'-a'\rangle+\langle x_2^*,x'-b'\rangle \le \max\{\norm{x'-a'},\norm{x'-b'}\} < \norm{x-a}-\rho\al_1.
\end{equation}

Third, by the Cauchy-Schwarz inequality and the first inequality of \eqref{P5.13P1}, it holds that
\begin{equation}\label{l2:Thao_est2}
	\langle x_1^*+x_2^*,x-x'\rangle \le \rho\al_2.
\end{equation}

Fourth, by the Cauchy-Schwarz inequality and \eqref{P5.13P1}, it holds that
	\begin{equation}\label{l2:Thao_est1_1}
	\langle x_1^*-a^*,a'-a\rangle+\langle x_2^*-b^*,b'-b\rangle \le \eps\norm{a'-a} + \eps\norm{b'-b} \le 2\la\eps.
\end{equation} 
Thanks to \eqref{P5.13P2}, it also holds that
	\begin{equation}\label{l2:Thao_est1_2}
	\ip{a^*}{a'-a} + \ip{b^*}{b'-b} \le \frac{\be}{2}\norm{a'-a} + \frac{\be}{2}\norm{b'-b} \le \frac{\be}{2}\la + \frac{\be}{2}\la
	= \la\be.
\end{equation}
Adding \eqref{l2:Thao_est1_1} and \eqref{l2:Thao_est1_2} yields that
	\begin{equation}\label{l2:Thao_est1}
	\langle x_1^*,a'-a\rangle+\langle x_2^*,b'-b\rangle \le 2\la\eps + \la\be.
\end{equation}
 
Making use of \eqref{l2:Thao_est3}, \eqref{l2:Thao_est4}, \eqref{l2:Thao_est2}, \eqref{l2:Thao_est1}, $\la=(\al+1/\sqrt{\eps})\rho$ and \eqref{al-al3} successively, we come up with
	\begin{eqnarray*}
	\norm{x-a}&=&\langle x_1^*,x-a\rangle+\langle x_2^*,x-b\rangle\\
	&=&\langle x_1^*,x'-a'\rangle+\langle x_2^*,x'-b'\rangle+\langle x_1^*+x_2^*,x-x'\rangle+\langle x_1^*,a'-a\rangle+\langle x_2^*,b'-b\rangle\\
	&<& \norm{x-a}-\rho\al_1+\rho\al_2+2\la\eps + \la\be\\
	&=& \norm{x-a} - \rho \left(\al_1-\al_2-({2\eps+\be})({\al+1/\sqrt{\eps}})\right)<\norm{x-a},
	\end{eqnarray*}
which is a contradiction and hence the proof is complete.
	\qed
\end{proof}
\bigskip

We now state and prove the key estimates for our analysis.

\begin{theorem}[key estimates]\label{t:primal result}
	Let $\{A,B\}$ be a pair of closed subsets of $X$ and $\bx\in A\cap B$ and consider the following statements.
	\begin{enumerate}
		\item\label{t4_i} 
		$\{A,B\}$ is {\it intrinsically transversal at} $\bx$, that is, $\itr >0$.

		\item\label{primal itr} There are numbers $\al\in ]0,1[$ and $\eps>0$ satisfying that for every $a\in (A\setminus B)\cap \B_\eps(\bx)$, 	$b\in (B\setminus A)\cap \B_\eps(\bx)$ and $x\in \B_\eps(\bx)$ with $\norm{x-a}=\norm{x-b}$ and $\delta >0$, there is $\rho\in ]0,\delta[$ such that
	\begin{equation}\label{D6.2}
		\dist\paren{\paren{A\cap \overline\B_\la(a)} \times \paren{B\cap \overline\B_\la(b)},\parenn{\B_\rho(x)}_2} + \al\rho \le \norm{x-a},
		\end{equation}
		where $\la:=(\al+1/\sqrt{\eps})\rho$.
		\item\label{t4_iii} $\itrc >0$.
	\end{enumerate}
Then, it holds that \eqref{t4_i} $\Rightarrow$ \eqref{primal itr} $\Rightarrow$ \eqref{t4_iii}.
\bigskip

Furthermore, let $\itrpp$ denote the exact upper bound of all $\al>0$ satisfying condition \eqref{primal itr} for some $\eps>0$, then it holds that
	\begin{equation}\label{T4.1}
\itr \le \itrpp \le \itrc.
	\end{equation}
\end{theorem}

\begin{proof}
By the definition of $\itrpp$, condition \eqref{primal itr} of Theorem \ref{t:primal result} is equivalent to $\itrpp>0$.
Hence, the statement that \eqref{t4_i} $\Rightarrow$ \eqref{primal itr} $\Rightarrow$ \eqref{t4_iii} is a direct consequence of \eqref{T4.1} which will be proved in the remainder of this proof.

	We prove the first inequality of \eqref{T4.1}: $\itr \le \itrpp$.
	Since the inequality becomes trivial when $\itrpp\ge 1$.
	We will prove the inequality when $\itrpp< 1$ by taking an arbitrary number $\al$ satisfying $\itrpp<\al \le 1$ and showing that $\itr\le \al$.
	By the definition of $\itrpp$, for any $\eps>0$, there are points $a\in (A\setminus B)\cap \B_\eps(\bx)$, $b\in (B\setminus A)\cap \B_\eps(\bx)$ and $x\in \B_\eps(\bx)$ with $\norm{x-a}=\norm{x-b}$ and $\delta >0$ such that the inequality
	\begin{equation}
	\label{D6.3}
	\dist\paren{\paren{A\cap \overline\B_\la(a)} \times \paren{B\cap \overline\B_\la(b)},\parenn{\B_\rho(x)}_2} + \al\rho > \norm{x-a},
	\end{equation}
	holds true for all $\rho\in ]0,\delta[$, where $\la:=(\al+1/\sqrt{\eps})\rho$.
	
	Since $A\cap B$ is a closed set and $a,b\notin A\cap B$, there is a number $\gamma>0$ such that
	\begin{equation}\label{non}
	\B_\gamma(a)\cap (A\cap B)=\emptyset,\quad \B_\gamma(b)\cap (A\cap B)=\emptyset.
	\end{equation}
	Let us consider $\eps\in \,]0,1[$ and choose a number $\rho>0$ and satisfying
\begin{equation}\label{t4:rho}
\rho< \min\left\{\delta,\; \frac{\gamma}{\al+1/\sqrt{\eps}},\; \eps,\; \frac{\eps\norm{x-a}}{\al+1+1/\sqrt{\eps}}\right\}.
\end{equation}
Define $\eps':= \al\rho>0$ and noting from \eqref{t4:rho} that
\[
\eps' = \al\rho < \paren{\al+1+1/\sqrt{\eps}}\rho < \eps\norm{x-a} < \norm{x-a} = \max\{\norm{x-a},\norm{x-b}\}.
\]

Thanks to \eqref{D6.3}, we can apply Lemma~\ref{C3.2} for the sets $A\cap \overline\B_\la(a)$, $B\cap\overline \B_\la(b)$, and points $a\in A\cap \overline\B_\la(a)$, $b\in B\cap\overline{\B}_\la(b)$, $x\in X$ and constants $\rho$, $\eps':=\al\rho$ (in place of $\eps$ in Lemma~\ref{C3.2}), $\la$, $\tau=\frac{\la - 2\al\rho}{\la +\al\rho}$ to find points $\hat a \in A\cap \B_{\la}(a)$, $\hat b\in B\cap \B_{\la}(b)$, $\hat x\in \B_\rho(x)$ and vectors $x_1^*,x_2^*\in X^*$ such that \eqref{C3.1.3} holds, and
	\begin{gather}\label{est_1}
		\norm{x_1^*}+\norm{x_2^*}=1,\quad\norm{x_1^*+x_2^*}<\frac{\eps'}{\rho}=\al,\\
		\label{est_2}
		x_1^*\in N_{A\cap \overline\B_\la(a)}(\hat a) = N_A(\hat a),\quad x_2^*\in N_{B\cap\overline \B_\la(b)}(\hat b)=N_B(\hat b).
	\end{gather}
	
	We next make several observations regarding $\hat{a}$, $\hat{b}$, $\hat{x}$, $x_1^*$ and $x_2^*$.
First,	since $\la<\gamma$ thanks to the choice of $\rho$ at \eqref{t4:rho}, we have $\hat a\in A\cap \B_\gamma(a)$ which together with \eqref{non} implies that $\hat a\notin B$.
	That is, $\hat a\in A\setminus B$ and similarly, $\hat b\in B\setminus A$. 
	
Second, by the triangle inequality and $\rho<\eps$ at \eqref{t4:rho}, it holds that
	\begin{equation*}
\begin{aligned}
\norm{\bar x -\hat x} &\le \norm{\bar x-x}+\norm{x-\hat x}< \eps+\rho\le 2\eps,
\\
\norm{\bar x -\hat a} &\le \norm{\bar x-a}+\norm{a-\hat a}< \eps+(1/\sqrt{\eps}+\al)\rho\le \eps(1+\al)+\sqrt{\eps},
\\
\norm{\bar x -\hat b} &\le \norm{\bar x-b}+\norm{b-\hat b}< \eps+(1/\sqrt{\eps}+\al)\rho\le \eps(1+\al)+\sqrt{\eps}.
\end{aligned}
	\end{equation*}
This implies that $\hat x \to \bar x$, $\hat a \to \bx$, $\hat b\to \bx$ as $\eps\downarrow 0$.

Third, due to \eqref{est_1} and $\al\le 1$, we have that $x_1^*\neq 0$ and $x_2^*\neq 0$ since if otherwise, we would have $\|x_1^*+x_2^*\|=1\ge \al$, a contradiction to \eqref{est_1}.

Fourth, by the triangle inequality and $\rho < \frac{\eps\norm{x-a}}{\al+1+1/\sqrt{\eps}}$ at \eqref{t4:rho}, it holds that
\[
\norm{\hat x -\hat a} \le \norm{x-a}+\norm{\hat x-x}+\norm{a-\hat a}\le \norm{x-a}+(1+\al+1/\sqrt{\eps})\rho \le (1+\eps)\norm{x-a}.
\]
By similar estimates, we also get that
\begin{gather*}
(1-\eps)\norm{x-a} \le 	\norm{\hat x -\hat a} \le (1+\eps)\norm{x-a},\\
(1-\eps)\norm{x-b} \le 	\|\hat x -\hat b\| \le (1+\eps)\norm{x-b}.
\end{gather*}
This together with noting that $\norm{x-a}=\|x-b\|$ implies that
	\begin{equation}\label{1}
		\frac{1-\eps}{1+\eps}\le \frac{\norm{\hat x- \hat a}}{\|\hat x-\hat b\|}\le \frac{1+\eps}{1-\eps},
	\end{equation}
which implies that $\frac{\norm{\hat x-\hat a}}{\|\hat x-\hat b\|}\to 1$ as $\eps\downarrow 0$.

Fifth,	by the Cauchy-Schwarz inequality, \eqref{C3.1.3} and the definition of $\tau$, we have that
	\begin{eqnarray*}
	1=\norm{x_1^*}+\norm{x_2^*}&\ge&\frac{\langle x_1^*,\hat x -\hat a \rangle}{\norm{\hat x-\hat a}}+\frac{\langle x_2^*,\hat x -\hat b \rangle}{\|\hat x-\hat b\|}\\
	&\ge& 
	\frac{\langle x_1^*,\hat x -\hat a \rangle}{\max
		\left\{\norm{\hat x -\hat a},\|\hat x-\hat b\|\right\}}+\frac{\langle x_2^*,\hat x -\hat b \rangle}{\max
		\left\{\norm{\hat x -\hat a},\|\hat x -\hat b\|\right\}}\\
	&> &\tau = \frac{1/\sqrt{\eps}-\al}{1/\sqrt{\eps}+\al},
	\end{eqnarray*}
	which tends to $1$ as $\eps \downarrow 0$.
	Thus,
$$
\frac{\langle x_1^*,\hat x -\hat a \rangle}{\norm{\hat x-\hat a}}+\frac{\langle x_2^*,\hat x -\hat b \rangle}{\norm{\hat x-\hat b}}\to 1\quad \text{as}\quad \eps \downarrow 0.
$$
Due to the Cauchy-Schwarz inequality and $\norm{x_1^*}+\norm{x_2^*}=1$, the above convergence happens if and only if
$$
\frac{\langle x_1^*,\hat x -\hat a \rangle}{\norm{x_1^*}\norm{\hat x-\hat a}}\to 1 \quad \text{as}\quad \frac{\langle x_2^*,\hat x -\hat b \rangle}{\norm{x_2^*}\norm{\hat x-\hat b}}\to 1 \quad \text{as}\quad \eps \downarrow 0.
$$

In view of \eqref{est_1}, \eqref{est_2} and the five observations above, by letting $\eps \downarrow 0$ and comparing the definition \eqref{itr constant} of $\itr$, we obtain that $\itr \le \al$ as claimed.
\bigskip

We now prove the second inequality of \eqref{T4.1}: $\itrpp \le \itrc$.
If $\itrc=1$, the inequality becomes trivial.
Otherwise,	let us take any number $\al$ satisfying $\itrc<\al\le 1$ and prove that $\itrpp<\al$.
Choose a number $\al_1$ satisfying $\itrc < \al_1 <\al$.
Then choose also a number $\eps>0$ such that 
\begin{equation}\label{choose eps}
{2\eps}({\al+1/\sqrt{\eps}}) < \al-\al_1.
\end{equation}
Such a number $\eps$ exists since ${2t}({\al+1/\sqrt{t}})\to 0$ as $t\downarrow 0$.

By the construction \eqref{itrc constant} of $\itrc$,  there exist $a\in (A\setminus B)\cap \B_\eps(\bx)$, $b\in (B\setminus A)\cap \B_\eps(\bx)$, $x\in \B_\eps(\bx)$ with $\norm{x-a}=\norm{x-b}$ and $x_1^*,x_2^*\in X^*$ such that 
		\begin{gather}
		\dist(x_1^*,N_A(a))<\eps,\quad \dist(x_2^*,N_B(b))<\eps,\quad \norm{x_1^*}+\norm{x_2^*}=1,\quad \norm{x_1^*+x_2^*}<\alpha_1,\\
		\langle x_1^*,x-a\rangle=\norm{x_1^*}\norm{x-a},\quad \langle x_2^*,x-b\rangle=\norm{x_2^*}\norm{x-b}.
		\end{gather}
Then by Lemma~\ref{P3.2} with noting that $\al_1+{2\eps}({\al+1/\sqrt{\eps}})<\al$ by \eqref{choose eps}, there exists a number $\delta>0$ such that for any $\rho\in \,]0,\delta[$, \eqref{D6.3} holds true.
Hence, by the definition of $\itrpp$, we obtain
\[
\itrpp\le \al_1+{2\eps}({\al+1/\sqrt{\eps}}) <\al
\]
 as claimed.
 The proof is complete.
	\qed
\end{proof}

\begin{remark}
	The first estimate $\itr \le \itrpp$ holds true in Asplund spaces, and the second one $\itrpp\le \itrc$ holds true in general normed linear spaces.
\end{remark}

The first estimate of \eqref{T4.1} shows that the primal space property \eqref{primal itr} of Theorem \ref{t:primal result} is a necessary condition for intrinsic transversality, while the second one shows that the property is a sufficient condition for the property characterized by $\itrc$.
A combination of Theorem \ref{t:primal result} and Theorem \ref{t:question 3} bridges the gap between the these properties in the Hilbert space setting.
As a result, we establishes the primal space characterization of intrinsic transversality for the first time.

\begin{corollary}[primal characterization of intrinsic transversality]\label{c1:primal result}
	Let $\{A,B\}$ be a pair of closed subsets of $X$ and $\bx\in A\cap B$.
		Then the following two statements hold true.
	\begin{enumerate}
		\item\label{cor1_case1} If $\itrc$ is strictly greater than $1/\sqrt{2}$, then so are the constants $\itr$, $\itrd{w}$ and $\itrpp$.
		\item\label{cor1_case2} If $\itrc\le 1/\sqrt{2}$, then it holds that
	\begin{equation}\label{c1:T4.1} 
	\itr = \itrd{w} = \itrpp = \itrc.
	\end{equation}
	\end{enumerate}
\end{corollary}
\begin{proof}
The first statement \eqref{cor1_case1} easily follows from Remark \ref{r:itrc<} and the first estimate of \eqref{primal itr}, while the second one \eqref{cor1_case2} directly follows from Corollary \ref{c:equivalence} and \ref{primal itr}.
\end{proof}

An important implication of Corollary \ref{c1:primal result} is that the property \eqref{primal itr} of Theorem \ref{t:primal result} is indeed a primal space characteriztion of intrinsic transversality as desired.
Moreover, the equality $\itr=\itrpp$ at \eqref{c1:T4.1} completes the quantitative relationship between the two primal and dual space counterparts in the case of most interest.


\addcontentsline{toc}{section}{References}
\begin{thebibliography}{10}
\bibitem{BakDeuLi05} Bakan, A., Deutsch, F., Li, W.: Strong {CHIP}, normality, and linear regularity of convex sets. Trans. Amer. Math. Soc. \textbf{357}(10), 3831--3863 (2005)

\bibitem{BauBor93}
Bauschke, H.H., Borwein, J.M.: On the convergence of von {N}eumann's alternating projection algorithm for two sets. Set-Valued Anal.
\textbf{1}(2), 185--212 (1993)

\bibitem{BauBor96}
Bauschke, H.H., Borwein, J.M.: On projection algorithms for solving convex feasibility problems. SIAM Rev. \textbf{38}(3), 367--426 (1996)	

\bibitem{BauBorLi99} Bauschke, H.H., Borwein, J.M., Li, W.: Strong conical hull intersection property, bounded linear regularity, {J}ameson’s property {(G)}, and error bounds in convex optimization. Math. Program., Ser. A \textbf{86}(1), 135--160 (1999)

\bibitem{BauBorTse00} Bauschke, H.H., Borwein, J.M., Tseng, P.: Bounded linear regularity, strong {CHIP}, and {CHIP} are distinct properties. J. Convex Anal. \textbf{7}(2), 395--412 (2000)

\bibitem{BauLukPhaWan13a}
Bauschke, H.H., Luke, D.R., Phan, H.M., Wang, X.:
Restricted normal cones and the method of alternating projections: theory.
Set-Valued Var. Anal. \textbf{21}(3), 431--473 (2013)

\bibitem{BivKraRib18}
Bivas, M., Krastanov, M., Ribarska, N.: On tangential transversality. Preprint,
arXiv:\textbf{1810.01809}, 1-17 (2018)

\bibitem{BorLew00} Borwein, J.M., Lewis, A.S.: Convex Analysis and Nonlinear Optimization. Springer-Verlag, New York (2000)

\bibitem{BoyVan04} Boyd, S., Vandenberghe, L.: Convex Optimization. Cambridge University Press, Cambridge (2004)


\bibitem{BuiKru}
Bui, H.T., Kruger, A.Y.: Extremality, stationarity and generalized separation of collections of sets. J. Optim. Theory Appl. (2019)	

\bibitem{Cla83}
Clarke, F.H.: Optimization and Nonsmooth Analysis. John Wiley \& Sons Inc., New York (1983)	

\bibitem{DonRoc14}
Dontchev, A.L., Rockafellar, R.T.:
Implicit Functions and Solution Mappings. A View from Variational Analysis, second edn.
Springer-Verlag, New York (2014)

\bibitem{DruIofLew15}
Drusvyatskiy, D., Ioffe, A.D., Lewis, A.S.:
Transversality and alternating projections for nonconvex sets. Found. Comput. Math. \textbf{15}(6), 1637--1651 (2015)	

\bibitem{DruLew18}
Drusvyatskiy, D., Lewis, A.S.: Inexact alternating projections on nonconvex sets. Preprint, 
arXiv:\textbf{1811.01298}, 1--15 (2018)

\bibitem{GuiPol74}
Guillemin, V., Pollack, A.: Differential Topology. Prentice-Hall, Inc., Englewood Cliffs, N.J. (1974)

\bibitem{HesLuk13}
Hesse, R., Luke, D.R.: Nonconvex notions of regularity and convergence of fundamental algorithms for feasibility
problems. SIAM J. Optim. \textbf{23}(4), 2397--2419 (2013)

\bibitem{Hir76}
Hirsch, M.: Differential Topology. Springer Verlag, New York (1976)

\bibitem{Iof00}
Ioffe, A.D.:
Metric regularity and subdifferential calculus.
Russian Math. Surveys. \textbf{55}, 501--558 (2000)

\bibitem{Iof16}
Ioffe, A.D.: Metric regularity – a survey. Part I. Theory. J. Aust. Math. Soc. \textbf{101}(2), 188--243 (2016)

\bibitem{Iof16.2}
Ioffe, A.D.: Metric regularity – a survey. Part II. Applications. J. Aust. Math. Soc. \textbf{101}(3), 376--417 (2016)
	
\bibitem{Iof17.1}
Ioffe, A.D.: Transversality in variational analysis. J. Optim. Theory Appl. 174(2), 343--366 (2017)

\bibitem{Iof17}
Ioffe, A.D.: Variational Analysis of Regular Mappings. Theory and Applications. Springer Monographs in Mathematics.
Springer (2017)	

\bibitem{KhaKruTha15}
Khanh, P.Q., Kruger, A.Y., Thao, N.H.: On induction theorem and nonlinear regularity models. SIAM J. Optim. \textbf{25}(4), 2561--2588 (2015)

\bibitem{Kru03}
Kruger, A.Y.: On Fr\'echet subdifferentials. J. Math. Sci. \textbf{116}(3), 3325--3358 (2003)

\bibitem{Kru05}
Kruger, A.Y.: Stationarity and regularity of set systems. Pac. J. Optim. \textbf{1}(1), 101--126 (2005)	

\bibitem{Kru06}
Kruger, A.Y.: About regularity of collections of sets. Set-Valued Anal. \textbf{14}(2), 187--206 (2006)

\bibitem{Kru09}
Kruger, A.Y.: About stationarity and regularity in variational analysis. Taiwanese J. Math. \textbf{13}(6A), 1737--1785 (2009)	
	
\bibitem{Kru15}
Kruger, A.Y.: Error bounds and metric subregularity. Optimization. \textbf{64}(1), 49--79 (2015)	

\bibitem{Kru15.2} Kruger, A.Y.: Error bounds and H\"older metric subregularity. Set-Valued Var. Anal. \textbf{23}(4), 705--736 (2015)

\bibitem{Kru16} Kruger, A.Y.: Nonlinear metric subregularity. J. Optim. Theory Appl. \textbf{171}(3), 820--855 (2016)

\bibitem{Kru18}
Kruger, A.Y.: About intrinsic transversality of pairs of sets. Set-Valued Var. Anal. \textbf{26}(1), 111--142 (2018)	

\bibitem{KruLop12.1}
Kruger, A.Y., L\'opez, M.A.: Stationarity and regularity of infinite collections of sets.
J. Optim. Theory Appl. \textbf{154}(2), 339--369 (2012)

\bibitem{KruLop12.2}
Kruger, A.Y., L\'opez, M.A.: Stationarity and regularity of infinite collections of sets. Applications to infinitely constrained optimization.
J. Optim. Theory Appl. \textbf{155}(2), 390--416 (2012)	

\bibitem{KruLukNgu17}
Kruger, A.Y., Luke, D.R., Thao, N.H.: About subtransversality of collections of sets. Set-Valued Var. Anal. \textbf{25}(4),
701--729 (2017)	

\bibitem{KruLukNgu18}
Kruger, A.Y., Luke, D.R., Thao, N.H.: Set regularities and feasibility problems. Math. Program., Ser. B \textbf{168}(1), 1--33 (2018)
	
\bibitem{KruTha13}
Kruger, A.Y., Thao, N.H.: About uniform regularity of collections of sets. Serdica Math. J. \textbf{39}, 287--312 (2013)
	
\bibitem{KruTha14}
Kruger, A.Y., Thao, N.H.: About {$[q]$}-regularity properties of collections of
sets.
J. Math. Anal. Appl. \textbf{416}(2), 471--496 (2014)	

\bibitem{KruTha15}
Kruger, A.Y., Thao, N.H.: Quantitative characterizations of regularity properties of collections of sets. J. Optim. Theory Appl. \textbf{164}(1), 41--67 (2015)
	
\bibitem{KruTha16}
Kruger, A.Y., Thao, N.H.: Regularity of collections of sets and convergence of inexact alternating projections.
J. Convex Anal. \textbf{23}(3), 823--847 (2016)	

\bibitem{LewLukMal09}
Lewis, A.S., Luke, D.R., Malick, J.: Local linear convergence of alternating and averaged projections.
Found. Comput. Math. \textbf{9}(4), 485--513 (2009)
	
\bibitem{LewMal08}
Lewis, A.S., Malick, J.: Alternating projections on manifolds. Math. Oper. Res. \textbf{33}(1), 216--234 (2008)	

\bibitem{LiNgPon07}
Li, C., Ng, K.F., Pong, T.K.: The SECQ, linear regularity, and the strong CHIP for an infinite system of closed convex sets in normed linear spaces.
SIAM J. Optim. \textbf{18}(2), 643--665 (2007)	

\bibitem{LukTebNgu18}
Luke, D.R., Teboulle, M., Thao, N.H.: Necessary conditions for linear convergence of iterated expansive, set-valued mappings.
Math. Program. Ser. A, doi.org/10.1007/s10107-018-1343-8		

\bibitem{Mor06.1}
Mordukhovich, B.S.: Variational Analysis and Generalized Differentiation, I: Basic Theory. Grundlehren der mathematischen Wissenschaften. Springer-Verlag, New York (2006)		

\bibitem{NgZan07}
Ng, K.F., Zang, R.: Linear regularity and $\varphi$-regularity of nonconvex sets. J. Math. Anal. Appl. \textbf{328}, 257--280 (2007)	
	
\bibitem{NgaThe01}
Ngai, H.V., Th{\'e}ra, M.: Metric inequality, subdifferential calculus and applications. Set-Valued Anal. \textbf{9}(1-2), 187--216 (2001)
		
\bibitem{NolRon16}
Noll, D., Rondepierre, A.: On local convergence of the method of alternating projections. Found. Comput. Math.
\textbf{16}(2), 425--455 (2016)	
	
\bibitem{Pen13}
Penot, J.P.: Calculus without Derivatives. Graduate Texts in Mathematics. Springer, New York (2013)		

\bibitem{Pha16}
Phan, H.M.: Linear convergence of the Douglas-Rachford method for two closed sets. Optimization. \textbf{65}, 369--385 (2016)		

	
\bibitem{RocWet98}
Rockafellar, R.T., Wets, R.J.B.: Variational Analysis. Springer, Berlin (1998)
	
\bibitem{Tha15}
Thao, N.H.: About Regularity Properties in Variational Analysis and Applications in Optimization. PhD Thesis. Federation University Australia (2015)	

\bibitem{Tha18}
Thao, N.H.: A convergent relaxation of the Douglas-Rachford algorithm. Comput. Optim. Appl. \textbf{70}(3), 841--863 (2018)	
		
\bibitem{ThaSolVer18}
Thao, N.H., Soloviev, O., Verhaegen, M.: Convex combination of alternating projection and
Douglas--Rachford algorithm for phase retrieval. Submitted reprint (2018)			
		
\bibitem{YeYe97}
Ye, J.J., Ye, X.Y.: Necessary optimality conditions for optimization problems with variational inequality constraints.
Math. Oper. Res. \textbf{22}(4), 977--997 (1997)

\bibitem{ZheNg08}
Zheng, X.Y., Ng, K.F.: Linear regularity for a collection of subsmooth sets in Banach spaces. SIAM J. Optim. \textbf{19}(1), 62--76 (2008)
\end{thebibliography}
\end{document}